\documentclass[bjps]{imsart}

\RequirePackage[authoryear]{natbib}
\RequirePackage[colorlinks,citecolor=blue,urlcolor=blue]{hyperref}
\RequirePackage{graphicx}
\RequirePackage[english]{babel}

\RequirePackage{amsmath, latexsym, amsthm, amsfonts,amssymb}
\RequirePackage{algorithm}
\RequirePackage{algpseudocode}
\RequirePackage{epsfig}
\RequirePackage{extarrows}
\RequirePackage{float}
\RequirePackage{multirow}
\RequirePackage{caption}
\RequirePackage{subfigure}
\RequirePackage{booktabs}

\renewcommand{\th}{\theta}

\newcommand{\eps}{\varepsilon}

\renewcommand{\phi}{\varphi}

\newcommand{\scr}[1]{{\mathcal #1}}

\newcommand{\EE}{\mathbb{E}}

\newcommand{\PP}{\mathbb{P}}
\newcommand{\ind}{\mathbf{1}}

\newcommand{\Bm}{\begin{bmatrix}}
	\newcommand{\Em}{\end{bmatrix}}

\newcommand{\Th}{\Theta}

\newcommand{\simind}{\stackrel{\rm ind}{\sim}}
\newcommand{\simiid}{\stackrel{\rm iid}{\sim}}

\newcommand{\cit}{\cite}

\startlocaldefs
\numberwithin{equation}{section}
\theoremstyle{plain}
\newtheorem{thm}{Theorem}[section]

\newtheorem{lem}[thm]{Lemma}

\newtheorem{rem}[thm]{Remark}
\newtheorem{ex}[thm]{Example}

\endlocaldefs

\begin{document}
	
\begin{frontmatter}
\title{Nonparametric Bayesian estimation of a concave distribution function with mixed interval censored data}

\runtitle{Nonparametric Bayesian estimation of a concave distribution  under censoring}

\begin{aug}
	\author{\fnms{Geurt} \snm{Jongbloed}\thanksref{a}\ead[label=e1]{g.jongbloed@tudelft.nl}},
	\author{\fnms{Frank} \snm{van der Meulen}\thanksref{a}\ead[label=e2]{f.h.vandermeulen@tudelft.nl}}
	\and
	\author{\fnms{Lixue} \snm{Pang}\thanksref{a}%
		\ead[label=e3]{l.pang@tudelft.nl}}
	
	\runauthor{G. Jongbloed et al.}
	
	\affiliation[a]{Institute of Applied Mathematics,
		Delft University of Technology}					
	
	\address{Address of the First, Second and Third authors\\
		Van Mourik Broekmanweg 6, 2628 XE Delft, The Netherlands.\\
		\printead{e1,e2,e3}}

\end{aug}

\begin{abstract}
Assume we observe a finite number of inspection times together with information on whether a specific event has occurred before each of these times. Suppose replicated measurements are available on multiple event times.  The set of inspection times, including the number of inspections, may be different for each event. This is known as  mixed case interval censored data. We consider Bayesian estimation of the distribution function of the event time while assuming it is concave. We provide sufficient conditions on the prior such that the resulting procedure is consistent from the Bayesian point of view.
We also provide computational methods for drawing from the posterior and illustrate the performance of the Bayesian method in both a simulation study and two real datasets.

\end{abstract}


\end{frontmatter}

\section{Introduction}
\label{sec:intro}
In survival analysis, one is interested in the time a certain event occurs. For example, the event may be the onset of a disease. A well known complication often encountered in practice is censoring, where the precise time at which
an event occurs is unknown, but partial information on it is available. In right censoring for example, one only observes the event if it occurs before a certain censoring time, otherwise one observes the censoring time accompanied by the information that the event occurred after this time. In interval censoring, one never sees the exact event time. Only an interval of positive length (possibly infinite) is observed which contains the event time of interest.

Suppose $X$ models the actual event time for one subject. Instead of observing $X$ directly, we observe a finite number of {\it inspection times} $0<t_1<t_2<\dots<t_{k}<\infty$, together with the information which of the intervals $(t_{j-1},t_j]$ contains $X$.  We will assume a setting in which we obtain data that are modelled as independent and identically distributed realisations of $X_1,\ldots, X_n$, each of which is distributed as $X$. For each subject, the set of inspection times, as well as the number of inspections, may be different.
This type of data is known as {\it  mixed-case interval censored data}. Our model includes both the interval censoring case 1  model (also known as current status model) and interval censoring case 2 model  for which $k=1$ and $k=2$ respectively.
In many statistical models, there are reasons to impose specific assumptions on functional parameters, for example shape constraints. 
Incorporating such constraints into the estimation procedure  often improves the accuracy of the resulting estimator.
In this paper, we consider the problem of estimating the distribution function $F$ of $X$, assuming that $F$ is concave. 

\subsection{Related literature}
In \cit{GroenWeln}, the pointwise asymptotic distribution  of the maximum likelihood estimator (mle) of the distribution function in the interval censoring case 1 model is derived. For interval censoring case 2, the asymptotic pointwise distribution of
the mle is still not known.
In the mixed case interval censoring model, the mle  has been studied by \cit{SchYu} where it is shown to be  $L_1$-consistent.
In \cit{WelnZha} a panel count model is considered, which includes the mixed case interval censoring model as a special case, namely when the counting process has only one jump. For this panel count model, \cit{WelnZha} study
two estimators. In case the counting process has only one jump and there is one inspection time, their estimators coincide with the mle for current status data ($k=1$). If $k>1$, this is not the case.
\cit{DumJong04} consider the current status model with the additional constraint that the underlying distribution function $F_0$ is concave. It is shown that the supremum distance between the nonparametric least squares
estimator and the underlying distribution function $F_0$ is of order $(\log n/n)^{2/5}$. For mixed case interval censoring, the MLE is shown to be asymptotically consistent under the assumption that $F_0$ is concave or
convex-concave in \cit{DumJong06}. In addition, an algorithm  for computing the mle is proposed there.

\medskip
From the Bayesian perspective, \cit{SuVan} derived a nonparametric Bayesian estimator for the event time distribution function based on right-censored data, using the Dirichlet process prior. A special feature
in this right-censoring model is that the posterior mean estimator can be constructed explicitly. For interval censored data, this explicit construction is not available.
\cit{CalGom} propose  a  nonparametric Bayesian approach in the interval censoring model and use a Markov Chain Monte Carlo algorithm to obtain estimators for the posterior mean.
\cit{DosFre} consider the  Dirichlet Process  prior in the interval censoring model.  They develop and compare various  Monte Carlo based algorithms for computing Bayesian estimators.
 A host of closely related Bayesian nonparametric models have been implemented in the DP-package in the R-language, Cf.\ \cit{Jara2011}. 

\subsection{Contribution}

In this paper, we define and study a Bayesian estimator of the event time distribution based on {\it  mixed-case interval censored data} under the additional assumption that the distribution function is concave.
An advantage of the Bayesian setup is the ease of constructing credible regions. To construct frequentist analogues of these, confidence regions, can be quite cumbersome, relying on either bootstrap simulations or asymptotic arguments.
We address this problem from a theoretical perspective and provide conditions on the prior such that the resulting procedure is consistent. That is, assuming data are generated from a ``true'' distribution, we show that the posterior asymptotically (as the sample size increases) converges to this distribution.
The proof relies on Schwartz' method for proving posterior consistency (Cf.\ Section 6.4 in \cit{GhoVaart}).
In addition, we provide computational methods for drawing from the posterior and illustrate its performance in a simulation study. Finally, we apply the Bayesian procedure on two real data sets and construct pointwise credible sets.

\subsection{Outline}  Section \ref{sec:modelnotation} sets off with introducing notation and formally describing the model. In section \ref{sec:consistency} we derive posterior consistency
under a weak assumption on the prior distribution on the class of concave distribution functions.
A Markov Chain Monte Carlo algorithm for obtaining draws from the posterior using the Dirichlet Mixture Process prior is detailed  in section \ref{sec:alg}. 
In section \ref{sec:siml} we perform a simulation study to illustrate the behaviour of the proposed Bayesian method. Furthermore, we apply it to two data sets in section \ref{sec:realdata}, one concerned with Rubella and the other with breast cancer. The appendix contains proofs of some technical results.

\section{Model, likelihood and prior}
\label{sec:modelnotation}
\subsection{Model and likelihood}
Suppose $X$ is a random variable in $[0,\infty)$ with  concave distribution function $F_0$. Instead of observing $X$, we observe the random vector $(K,T,\Delta)$ that is constructed as follows. First, $K$ is sampled
from a discrete distribution with probability mass function $p_K$ on $\{1,2\ldots\}$, representing the number of inspection times. Given $K=k$, $T\in\mathbb{R}^k$ is sampled from a density $g_k$ supported on the set
$\{t=(t_1,\ldots, t_k) \in (0,L]^k\,:\,0<t_1<\cdots<t_k<\infty\}$ for some constant $L$. This random vector contains the (ordered) inspection times. Finally, $\Delta\in\{0,1\}^{k+1}$ is the vector indicating in which of the $k+1$
intervals generated by $T$ the event actually happened. Thus,  it is defined as the vector with $j$-th component
$$
\Delta_j=1_{(T_{j-1},T_{j}]}(X) \mbox{ for } 1\le j\le k+1
$$
where $T_0=0$ and $T_{k+1}=\infty$ by convention.

This procedure is repeated independently, so for sample size $n$ the data is a realisation of
$$
\mathcal{D}_n:=\{(K_i,T^{i},\Delta^{i})=(K_i,T_{i,1},\ldots,T_{i,K_i},\Delta_{i,1},\ldots,\Delta_{i,K_i+1}),\,  i=1,\dots,n\}.
 $$
 Define the sets
\begin{equation}
\label{eq:defCk}
\mathcal{C}_k=\{t\in (0,L]^k:\,0<t_1<\cdots<t_k<\infty\}
\end{equation}
 and $\mathcal{H}_k=\{\delta\in\{0,1\}^{k+1}:\,\sum_{j=1}^{k+1}\delta_j=1\}$, $k=1,2,\dots$. Then $\mathcal{D}_n\in\left(\bigcup_{k=1}^\infty\{k\}\times\mathcal{C}_k\times\mathcal{H}_k\right)^n$.

Upon conditioning on the observed inspection times, we can define the likelihood of the distribution function $F$ by
\begin{equation}\label{eq:likelih}L(F)=\prod_{i=1}^n \Big(p_{K}(K_i)g_{K_i}(T^{i})\prod_{j=1}^{K_i+1}(F(T_{i,j})-F(T_{i,j-1}))^{\Delta_{i,j}}\Big).\end{equation}
We denote the joint distribution of $\{(K_i, T^i),\, 1\le i \le n\}$  by $\PP_{K,T}$. Given these $(K_i,T^i)s$ the vectors $\Delta^i$ have multinomial distributions with probabilities depending on $F_0$. The  distribution of $\mathcal{D}_n$  will be denoted by $\PP_0$. Expectation with respect to measures will be denoted by $\EE$, supplemented by a subscript referring to the measure.

\subsection{Prior specification}
In order to estimate the underlying concave distribution function in a Bayesian way, we  construct a  prior distribution on the set of all concave distribution functions. For $\theta>0$, denote the
uniform density function on $[0,\theta]$ by $\phi(\cdot\mid\theta)$ and its distribution function by $\Psi(\cdot\mid\theta)$, i.e.
\begin{equation}\label{phi}\phi(x,\theta)=\frac1{\theta}1\{x\le \theta\} \mbox{ and }\Psi(x,\theta)=\frac{\min(x,\theta)}{\theta} \mbox{ respectively}, \quad x\ge0.\end{equation}
It is well known that any concave distribution function $F$ on $[0,\infty)$ allows the mixture representation (see \cit{feller})
\begin{equation}\label{eq:mixpre}F(x) = \int \Psi(x, \th) d G(\th),\end{equation}
where $G$ is a distribution function on $[0,\infty)$. In what follows, we sometimes stress this representation and denote the concave distribution function by $F_G$. In order to put a prior measure $\Pi$ on the set
$$\scr{F}=\Big\{F: F \ \text{is a concave distribution on}\ [0,\infty) \Big\},$$
we use (\ref{eq:mixpre}) together with a prior distribution $\Pi^\ast$ on the set of all mixing distribution functions $G$ on $(0,\infty)$ (denote as $\mathcal{M}$).
Having chosen such a prior measure, we denote the resulting posterior measure on $\mathcal{F}$ by $\Pi(\cdot|\mathcal{D}_n)$.

\section{Posterior consistency}
\label{sec:consistency}
In this section we establish consistency of the posterior distribution $\Pi(\cdot|\mathcal{D}_n)$ under a weak condition on the prior measure $\Pi$. Generally, the posterior is said to be consistent at $F_0$
(with respect to a semimetric $d$) if for any $\eps>0$, $\EE_0 \Pi(d(F,F_0) >\eps \mid \mathcal{D}_n) \to 0$ when $n\to \infty$.

For any distribution function $G$, denote $G_{i,j}=G(T_{i,j})-G(T_{i,j-1})$. Given the inspection times $\{T^{i},1\le i\le n\}$, we say that  distribution functions $G$ and $F$ belong to the same equivalence class
if the increments between the adjacent times are the same: $G_{i,j}=F_{i,j}$ for all $i=1,\dots,n$, $j=1,\dots,K_i+1$. Then given data $\mathcal{D}_n$, we define a distance $d$ between two (equivalence classes of)
distribution functions $G$ and $F$ by
\begin{equation}\label{lossfunc}d_n(G,F)=\frac1{n}\sum_{i=1}^n\sum_{j=1}^{K_i+1}\left| G_{i,j}-F_{i,j}\right|.\end{equation}
Recall that $\Pi^\ast$ is a prior on the set $\mathcal{M}$,
then $G$ is in the weak support of $\Pi^\ast$ if every weak neighborhood of $G$
 has positive measure.

\begin{thm}\label{thm:consis}
Fix $F_0 \in \scr{F}$ and $x\in [0,\infty)$. Consider the mixed-case interval censoring model described in section \ref{sec:intro}. Assume $F_0$ has a continuous density function $f_0$ on $(0,\infty)$ with
$f_0(0)\le M<\infty$ and that the weak support of the prior distribution $\Pi^\ast$ is $\mathcal{M}$. If $\EE K^r<\infty$, for some $r>1/2$,
then for any $\epsilon>0$, we have $\PP_0$-almost surely that
$$\Pi(F\in\mathcal{F}: d_n(F,F_0)>\epsilon|\,\mathcal{D}_n)\to 0\mbox{ as }n\to\infty.$$
\end{thm}

Note that $d_n$ in Theorem \ref{thm:consis} is a random semidistance since it depends on the inspection times $\{K_i,T^i, i=1,\dots,n\}$, also depending on $n$. Define the  measure $\mu$ on the Borel $\sigma-$field $\mathcal{B}$
on $[0,\infty)$ that  measures the ``expected proportion of inspection times contained in a Borel set $B\in{\cal B}$'' by 
\[\mu(B)=\sum_{k=1}^\infty p_K(k)k^{-1}\int g_k(t)\sum_{j=1}^k\mathbf{1}_B(t_j)dt.\]
As a special case, assume that given $k$,  $S_1,\dots,S_k$ are independent and identically distributed with density function $\xi$ on $[0,\infty)$ and $\{T_1<T_2<\dots<T_k\}$ are the ordered $S_j$'s. 
Then when $k=1$, \[\mu(B)= \int g_1(t_1)\mathbf{1}_B(t_1)dt_1=\int_B\xi(x)dx.\]
When $k=2$, for any $a\in [0,\infty)$
\begin{align*}
\mu((0,a])=&\frac1{2}\int g_2(t)(\mathbf{1}\{t_1\le a\}+\mathbf{1}\{t_2\le a\})dt
=\frac1{2}(\PP(t_1\le a)+\PP(t_2\le a))\\
&=\frac1{2}\left(1-\left(1-\int_0^a\xi(x)dx\right)^2+\left(\int_0^a\xi(x)dx\right)^2\right)=\int_0^a\xi(x)dx
\end{align*}
Hence,  the measure $\mu$ has density $\xi$ in interval-censoring cases 1 and 2.

The follow result establishes posterior consistency with respect to $L_1(\mu)$ loss.
\begin{thm}\label{corr:L1mu}
Let $F_0$, $\Pi$ and $K$ satisfy the conditions of Theorem \ref{thm:consis}. Then for any $\epsilon>0$, we have
\[\EE_0\,\Pi\left(F\in\mathcal{F}: \int \left| F-F_0 \right| d\mu>\epsilon\mid\mathcal{D}_n\right)\to 0\mbox{ as }n\to\infty.\]
\end{thm}

\subsection{Proofs}
For proving Theorem \ref{thm:consis} we  use  Schwartz' approach to derive  posterior consistency. In the proof of this theorem, Lemma \ref{lem:KL} is used to control the prior mass of a neighbourhood of the true distribution. Lemma \ref{lem:test} provides appropriate test functions. Both lemmas are stated below; the proofs are in appendix \ref{sec:app_paper2}.

\begin{lem}\label{lem:KL}
Let $F_0$ and $\Pi^\ast$ satisfy the conditions of Theorem \ref{thm:consis}. Define, for $F_1,F_2\in\mathcal{F}$, $k=1,2,\ldots$ and $t\in \mathcal{C}_k$ as defined in (\ref{eq:defCk}):
\begin{equation}\label{eq:h12}
h_{k,F_1,F_2}(t)=\sum_{j=1}^{k+1}(F_0(t_j)-F_0(t_{j-1}))\log\frac{F_1(t_j)-F_1(t_{j-1})}{F_2(t_j)-F_2(t_{j-1})}
\end{equation}
(where $t_0=0$ and $t_{k+1}=\infty$ by convention).
If we define,
\begin{equation}\label{eq:Sdelta}S(\eta)=\biggl\{F\in\scr{F}\: : \: \sum_{k=1}^\infty p_K(k)\int g_k(t)h_{k,F_0,F}(t)dt<\eta \biggr\}.\end{equation}
then for all $\eta>0$, $\Pi(S(\eta))>0$.
\end{lem}
Note that that for the specific choice $F_1=F_0$, by Jensen's inequality, $h_{k,F_0,F}\ge 0$ for all $F\in {\cal F}$.

\begin{lem}\label{lem:test}
For $\epsilon>0$, define $U_{\epsilon}:=\{F\in\mathcal{F}:d_n(F,F_0)>\epsilon\}$. Then there exists a sequence of test functions $\Phi_n$ such that for all $n\ge 1$,
\begin{equation}
\label{tests}
\begin{split}
\EE_0(\Phi_n)\le Ce^{-nc}\\
\EE_{(K,T)}\left\{\sup_{F\in U_{\epsilon}}\EE_F[1-\Phi_n|K,T]\right\}\le Ce^{-nc}
\end{split}\end{equation}
for some positive constants $c$ and $C$.
\end{lem}

\begin{proof}[Proof of Theorem \ref{thm:consis}]
Choose $\epsilon>0$ and define the set $U_{\epsilon}$ as in Lemma \ref{lem:test}.
Define
\[ Z_{ij} = \frac{F(T_{i,j})-F(T_{i,j-1})}{F_0(T_{i,j})-F_0(T_{i,j-1})}. \]
 Using expression (\ref{eq:likelih}) of the likelihood, the posterior mass of the set $U_{\epsilon}$ can be written as
\[\Pi(U_{\epsilon}\mid\mathcal{D}_n)=D_n^{-1}\int_{U_{\epsilon}}\prod_{i=1}^n\prod_{j=1}^{K_i+1}Z_{i,j}^{\Delta_{i,j}}d\Pi(F),\]
where
$$D_n=\int\prod_{i=1}^n\prod_{j=1}^{K_i+1}Z_{i,j}^{\Delta_{i,j}}d\Pi(F).$$
Fix $0<\eta<c/2$, where $c$ is as it appears in Lemma \ref{lem:test}. Also fix $F\in S(\eta)$.

We first show that Lemma \ref{lem:KL} implies for any $\eta'>\eta$ we have $\PP_0$-a.s.\  that
\[D_n\ge \exp(-n\eta')\Pi(S(\eta))\] for all $n$ sufficiently large.
By Lemma \ref{lem:KL}, we have $\Pi(S(\eta))>0$. Let $\Pi_{S(\eta)}$ be $\Pi$ restricted to $S(\eta)$ and normalised to a probability measure. For $i\ge 1$  define
\[	Y_{i,j}=-\int\Delta_{i,j}\log Z_{i,j} d\Pi_{S(\eta)}(F) \:\ind_{\{1,2,\ldots,K_i+1\}}(j).\] Note that,
\begin{align*}
\EE_0\left[\sum_{j=1}^{K_1+1}Y_{1,j}\right]&=\EE_{K_1,T_1}\left[\EE_{F_0}\left[\sum_{j=1}^{K_1+1} Y_{1,j}\mid T^{K_1},K_1\right]\right]\\
&=\EE_{K_1,T_1}\left[\sum_{j=1}^{K_1+1}\int -(F_0(T_{1,j})-F_0(T_{1,j-1}))\log Z_{i,j} d\Pi_{S(\eta)}(F)\right]\\
&=\sum_{k=1}^\infty p_K(k)\int\int g_k(t)h_{k,F_0,F}(t)dtd\Pi_{S(\eta)}(F)\le\eta<\infty.
\end{align*}
Therefore, the law of large numbers yields
\[\frac1{n}\sum_{i=1}^n\sum_{j=1}^{K_i+1}Y_{i,j}\to \EE_0\left[\sum_{j=1}^{K_1+1}Y_{1,j}\right]\le\eta, \qquad \PP_0-a.s.\]
Hence, $\PP_0$-a.s.\  for any $\eta'>\eta$,
\begin{align}
D_n &\ge\int_{S(\eta)}\prod_{i=1}^n\prod_{j=1}^{K_i+1}Z_{i,j}^{\Delta_{i,j}}d\Pi(F)
=\Pi(S(\eta))\int\prod_{i=1}^n\prod_{j=1}^{K_i+1}Z_{i,j}^{\Delta_{i,j}}d\Pi_{S(\eta)}(F)
\nonumber \\
&=\Pi(S(\eta))\int\exp\left(\sum_{i=1}^n\sum_{j=1}^{K_i+1}\Delta_{i,j}\log Z_{i,j}\right)\: d\Pi_{S(\eta)}(F)\nonumber\\
&\ge \Pi(S(\eta))\exp\left(-n\cdot\frac1n\sum_{i=1}^n\sum_{j=1}^{K_i+1}Y_{i,j}\right)\nonumber\\
\label{eq:lowerDn}
&\ge \exp(-n\eta')\,\Pi(S(\eta))
\end{align}
for $n$ sufficiently large, where we used Jensen's inequality in the second inequality.

Now we can finish the proof by combining this result with the test functions
 $\Phi_n$ satisfying (\ref{tests}) (by Lemma \ref{lem:test}).

By inequality (\ref{eq:lowerDn}), we can bound $\EE_0\Pi(U_{\epsilon}\mid\mathcal{D}_n)$ as follows,
\begin{align*}
&\EE_0\Pi(U_{\epsilon}\mid\mathcal{D}_n)=\EE_0\Pi(U_{\epsilon}\mid\mathcal{D}_n)\Phi_n+\EE_0\Pi(U_{\epsilon}\mid\mathcal{D}_n)(1-\Phi_n)\\
&\,\,\, \le\EE_0\Phi_n+
\Pi(S(\eta))^{-1}e^{n\eta'}\EE_0\int_{U_{\epsilon}}\prod_{i=1}^n\prod_{j=1}^{K_i+1}Z_{i,j}^{\Delta_{i,j}}(1-\Phi_n)d\Pi(F)\\
&\,\,\,=\EE_0\Phi_n+\Pi(S(\eta))^{-1}e^{n\eta'}\EE_{(K,T)}\int_{U_\epsilon}\EE_F(1-\Phi_n)d\Pi(F)\\
&\,\,\,\le Ce^{-cn}+\Pi(S(\eta))^{-1}\cdot Ce^{-(c-\eta')n}
= o(1) \qquad\text{as}\quad n\to\infty.
\end{align*}
The final step follows by choosing $\eta'<c$. Since $\sum_{n=1}^\infty e^{-bn}<\infty$ for any constant $b$, almost sure convergence follows by the Borel-Cantelli lemma.
\end{proof}

\medskip
\begin{proof}[Proof of Theorem \ref{corr:L1mu}]
First note that the proof of (16) in \cit{DumJong06} shows for all distribution functions $F,F_0\in\mathcal{F}$,
\begin{equation}\label{d'ndn}d'_n(F,F_0)=\frac1{n}\sum_{i=1}^nK_i^{-1}\sum_{j=1}^{K_i}\left| F(T_{i,j})-F_0(T_{i,j})\right| \le d_n(F,F_0).\end{equation}
For any $\epsilon>0$, denote set
\[A_n=\left\{\sup_{F\in\mathcal{F}}\left|d'_n(F,F_0)-\int|F-F_0|d\mu\right|>\epsilon/2\right\},\]
Now we prove that $\PP_{(K,T)}(A_n)\to 0$ as $n\to\infty$.
Fix $F_0\in\mathcal{F}$ and  denote
$$\psi_i(F)=n^{-1}K_i^{-1}\sum_{j=1}^{K_i}|F(T_{i,j})-F_0(T_{i,j})|.$$
Then $d'_n(F,F_0)=\sum_{i=1}^n\psi_i(F)$. Note that $\EE_{(K,T)} d'_n(F,F_0)=\int| F-F_0| d\mu$.
It is sufficient to show that
\begin{equation}\label{eq:supbound}\EE_{(K,T)}\sup_{F\in\mathcal{F}}|d'_n(F,F_0)-\EE_{(K,T)} d'_n(F,F_0)|\to 0.\end{equation}
By theorem \ref{thm:Pollard}, it is implied by the existence of a sequence $\delta_n\to 0$ such that
\begin{align}
&\EE_{(K,T)}\sum_{i=1}^n\sup_{F\in\mathcal{F}}|\psi_i(F)|=O(1),\label{eq:A1}\\
&\EE_{(K,T)}\sum_{i=1}^n \mathbf{1}\{\sup_{F\in\mathcal{F}}|\psi_i(F)|>\delta_n\}\sup_{F\in\mathcal{F}}|\psi_i(F)|=o(1),\label{eq:A2}\\
&\text{for any}\, u>0,\quad \log\mathcal{N}(u,\mathcal{F},\rho_n)=c(u).\label{eq:A3}
\end{align}
Here
\[\mathcal{N}(u,\mathcal{F},\rho_n)=\min\left\{\#\mathcal{G}: \mathcal{G}\subset\mathcal{F},\inf_{G\in\mathcal{G}}\rho_n(F,G)\le u\,\text{for all}\, F\in\mathcal{F}\right\},\]
and
\[\rho_n(F,F')=\sum_{i=1}^n|\psi_i(F)-\psi_i(F')|.\]
For $(\ref{eq:A1})$ and $(\ref{eq:A2})$, note that $\sup_{F\in\mathcal{F}}|\psi_i(F)|\le n^{-1}$, hence $\EE_{(K,T)}\sum_{i=1}^n\sup_{F\in\mathcal{F}}|\psi_i(F)|\le 1$.
By taking $n\delta_n\to\infty$, e.g. $\delta=\frac1{\sqrt{n}}$,
\[
\EE_{(K,T)}\sum_{i=1}^n1\{\sup_{F\in\mathcal{F}}|\psi_i(F)|>\delta_n\}\sup_{F\in\mathcal{F}}|\psi_i(F)|\le n^{-1}\EE\sum_{i=1}^n1\{n^{-1}>\delta_n\}
=1\{n^{-1}>\delta_n\}\to 0
\]
For $(\ref{eq:A3})$, note that 
\begin{align*}
\rho_n(F,F')&=\sum_{i=1}^n|\psi_i(F)-\psi_i(F')|
\le n^{-1}\sum_{i=1}^nK_i^{-1}\sum_{j=1}^{K_i}\left||F(T_{i,j})-F_0(T_{i,j})|-|F'(T_{i,j})-F_0(T_{i,j})|\right|\\
&\le n^{-1}\sum_{i=1}^nK_i^{-1}\sum_{j=1}^{K_i}|F(T_{i,j})-F'(T_{i,j})|
=\int|F-F'|d\upsilon
\le \left(\int|F-F'|^2d\upsilon\right)^{1/2}
\end{align*}
where the measure $\upsilon$ is defined by $\upsilon(\cdot)=n^{-1}\sum_{i=1}^nK_i^{-1}\sum_{j=1}^{K_i}\delta_{T_{i,j}}(\cdot)$. In the final step, we use H\"{o}lder's inequality and that  $\upsilon$ has total mass 1. Further, using Lemma 2.1 and equation (2.5) in
\cit{Geer} we obtain
\[\log\mathcal{N}(u,\mathcal{F},\rho_n)\le \log\mathcal{N}(u,\mathcal{F},L_2(\upsilon))\le Cu^{-1}\]
for some constant $C$ and any $u>0$.

Therefore, denote $B_\epsilon=\{F\in\mathcal{F}:\int|F-F_0|d\mu>\epsilon\}$, by $\PP_{(K,T)}(A_n)\to 0$ as $n\to\infty$ and inequality (\ref{d'ndn}), we have
\begin{align*}
\EE_0\Pi(B_\epsilon\mid\mathcal{D}_n)&=\EE_0\Pi(B_\epsilon\mid\mathcal{D}_n)\mathbf{1}_{A_n}+\EE_0\Pi(B_\epsilon\mid\mathcal{D}_n)\mathbf{1}_{A^c_n}\\
&\le \EE_0(1_{A_n})+\EE_0\Pi(F\in\mathcal{F}:d'_n(F,F_0)>\epsilon/2\mid\mathcal{D}_n)\\
&\le \PP_{(K,T)}(A_n)+\EE_0\Pi(F\in\mathcal{F}:d(F,F_0)>\epsilon/2\mid\mathcal{D}_n)\to0
\end{align*}
as $n\to\infty$.
\end{proof}

\section{Computational methods}
\label{sec:alg}
Assume the mixing measure $G$ is a Dirichlet process with base measure $G_0$ (with density $g_0$) and concentration rate $\alpha$. The prior distribution this induces on $\mathcal{F}$ through (\ref{eq:mixpre}) is called a
Dirichlet Mixture Process (DMP).  Denoting by  $\#(x)$  the number of distinct values in a vector $x$, a sample
$X_1,\ldots,X_n$ from the DMP can be generated using the following steps:
\begin{equation}
\label{eq:mod2}
\begin{split}
 Z:=(Z_1,\ldots, Z_n) &\sim \mathrm{CRP}(\alpha)\\
\Th_1,\ldots, \Th_{\#(Z)} &\simiid G_0 \\
X_1,\ldots, X_n \mid \Th_1,\ldots, \Th_{\#(Z)}, Z_1,\dots, Z_n& \simind \mathrm{Unif}(0, \Th_{Z_i}).
\end{split}
\end{equation}
Here  CRP$(\alpha)$ denotes the ``Chinese Restaurant Process''  that can be viewed as follows. Assume in a Chinese restaurant, the first customer sits at the first table. Then, given a number of occupied tables, the next
customer joins one of these tables with a probability proportional to the number of customers already there, or starts a new table with probability proportional to $\alpha$. Interpreting $Z_i$ as the number of customers
sitting at table $i$ after $n$ customer arrivals, this leads to a distribution on the space of partitions of the integers $\{1,2,\ldots,n\}$.

\medskip
In the interval censoring model, we do not observe the $X_i$'s, but for each $i$ the interval $(L_i,R_i]=(T_{i,J_i-1},T_{i,J_i}]$ that contains $X_i$. We are then interested in  the conditional distribution of $(Z,\Th)$ given
the data $\mathcal{D}_n$.  In case we would have complete observations $X_1,\dots,X_n$, there are algorithms to sample from this conditional distribution (see \cit{Neal}). Having only the interval censored data, we can adapt such
algorithms, treating the unobserved event times $X_i$ as latent variables in the same fashion as this is done in the case of right censoring by \cit{HanLau}. Given the exact values $X_i$, we can use existing algorithms to generate
samples from the posterior. Subsequently, we update the $X_i$'s in each iteration by sampling conditionally on the time intervals $(L_i,R_i]$ where the event happened.

\medskip
We initialise a Gibbs sampler by specifying values of  $(Z,\Th,X)$ that satisfy the constraints in the model. This means that $\Th_{Z_i}\ge X_i$ and $X_i\in(L_i,R_i]$ for $i=1,\dots,n$. For ease of notation let
$\Theta=(\Theta_1, \ldots, \Theta_{\#(Z)})$ and $X=(X_1,\dots,X_n)$ for $i=1,\dots,n$. Then the following steps are iterated:
\begin{enumerate}
  \item sample $Z \mid (X,\Th,\mathcal{D}_n)$;
  \item sample $\Th \mid (X,Z,\mathcal{D}_n)$;
  \item sample $X \mid (\mathcal{D}_n,\Th, Z)$.
\end{enumerate}
 Given $X$, $\mathcal{D}_n$ does not play any role when sampling $Z$ and $\Th$. Hence the first two steps are the same as in the case of precise observations. More details on this step in that setting can be found in \cit{JMP}
 and \cit{Neal}. The final step is to sample the latent variables $X$ given $\mathcal{D}_n$, $Z$ and $\Th$. For this, note that
\[f_{X_i \mid \mathcal{D}_n,\Th, Z}(x \mid \mathcal{D}_n,\th,z)\propto f(x \mid \th_{z_i})\mathbf{1}_{(L_i,R_i]}(x)=\phi(x \mid \th_{z_i})\mathbf{1}_{(L_i,R_i]}(x).\]

This is the density of the uniform distribution  on interval $(L_i,R_i]\cap [0,\Th_{Z_i}]$. Note that with the initialisation described above, $(L_i,R_i]\cap [0,\Th_{Z_i}]$ is non-empty.

\bigskip
Using the conjugacy property of Dirichlet process (see e.g. \cit{Fegsn}), the conditional expectation of the posterior of $F$ is given by
\[ \EE\left[ \int \Psi(x,\th) d G(\th)\, |\, \Th, Z, \mathcal{D}_n\right]
= \frac1{\alpha+n} \left(\alpha \int \Psi(x,\th) d G_0(\th) + \sum_{i=1}^n \Psi(x,\Th_{Z_i})\right). \]
Hence, the posterior mean of $F$ can be obtained using a Markov Chain Monte Carlo approximation of the posterior of $(\Th,Z)$ given $\mathcal{D}_n$. Having the algorithms to generate from the distribution of $(\Th,Z)\mid \mathcal{D}_n$,
assume in the $j$-th iteration we obtained $\left(\Th^{(j)}_{Z^{(j)}_1}, \ldots, \Th^{(j)}_{Z^{(j)}_n}\right)$.
At iteration $j$, a sample from the posterior is given by
\begin{equation}\label{eq:postmean} \hat{F}^{(j)}(x):
= \frac{\alpha}{\alpha+n} \int \Psi(x,\th) d G_0(\th) +\frac1{\alpha+n} \sum_{i=1}^n \Psi(x,\big(\Th^{(j)}_{Z^{(j)}_i}\big)). \end{equation}
 After $J$ iterations, an estimator for the posterior mean is given by $J^{-1}\sum_{j=1}^J \hat{F}^{(j)}(x)$.

\begin{rem}
	In case the Dirichlet process is truncated, the target density is of fixed dimension. One of the referees raised the question whether probabilistic programming languages such as JAGS, BUGS, Stan or Turing can be used. First of all, we do not  consider truncation here as, strictly speaking, it is not necessary. However, we fully  agree that from a practical point of view the proposed approach may be implemented using one of the suggested Bayesian computational packages in case of truncation.  What might be tricky here is that the workehorse algorithm in for example Stan (Hamiltonian Monte Carlo)  uses automatic differentiation for computing gradients. However, the density of the uniform distribution on $[0,\theta]$, viewed as a function of $\theta$ is not differentiable. 
	
	A host of related Bayesian nonparametric models have been implemented in the DP-package (Cf.\ \cit{Jara2011}). 
\end{rem}

\section{Simulation results}
\label{sec:siml}
In this section, we first study the posterior mean estimators of a concave distribution function based on simulated interval censored data. Next, we
compare the Bayesian and the frequentist methods in this setting.

We simulate data by repeating independently $n$ times the following scheme:
\begin{enumerate}
	\item sample $K$ from the discrete uniform distribution on the integers $\{1,\cdots,20\}$;
  \item sample $K$ inspection times $T_1<\dots<T_K$ by sorting $K$ independent and identically distributed random variables (we choose the  Gamma distribution with shape parameter equal to $2$ and rate parameter equal to $1$);
  \item sample $X$ from the standard Exponential distribution;
  \item set $L:=\sup_{j}\{T_j: T_j< X\}$ and $R:=\inf_{j}\{T_j: T_j\ge X\}$ (where $T_0=0$, $T_{k+1}=\infty$).
\end{enumerate}
This leads to the dataset $\mathcal{D}_n$ containing the observation intervals $(L_i,R_i]$ for  $1\le i\le n$.

The prior is specified by a Dirichtlet Process for the mixture measure. As seen in the formula (\ref{eq:postmean}), the concentration parameter $\alpha$ expresses our confidence on the prior.  As such the choice of $\alpha$ can be interpreted as a prior sample size. In the following we choose $\alpha=1$, expressing a small prior sample size. 

There is no obvious ``optimal'' choice for the base measure. The approach we take is motivated by the numerical study in \cite[section 5.1]{JMP}, where it is shown to yield a reasonable balance between computational tractability and performance.   Write $Y \sim Par(s,\xi)$ with $s>0$ and $\xi>0$ if $f_Y(y)= \xi s^\xi y^{-\xi-1} \ind\{y\ge s\}$. 
 We choose the base measure to be a mixture of $Par(s,1)$-distributions, where  $s \sim  Gamma(2,1)$. This implies that the density of the base measure, $g_0$, satisfies $g_0(\theta) \sim \theta$ for $\theta \approx 0$.  The Pareto-distribution is a conjugate prior to the Uniform distribution, alleviating computations in a Gibbs sampler. Mixing over $s$ is a practical way to robustify the prior. 
 
 For updating $S$, it follows from the short computation in section 5.1 in \cite{JMP}  that 
\[  f_{S\mid \Th,Z}(s\mid  \th, z) \propto s^{\#(z)} e^{-s} \ind\{s \le \wedge(\th)\}. \]

We take sample size $n=100$.
To show the algorithm's performance, we show a traceplot and autocorrelation function of $\hat{F}^j(1)$ over $30.000$ iterations in Figure \ref{fig:itacf}.

\begin{figure}[!htp]
\centering
{
\includegraphics[scale=0.45]{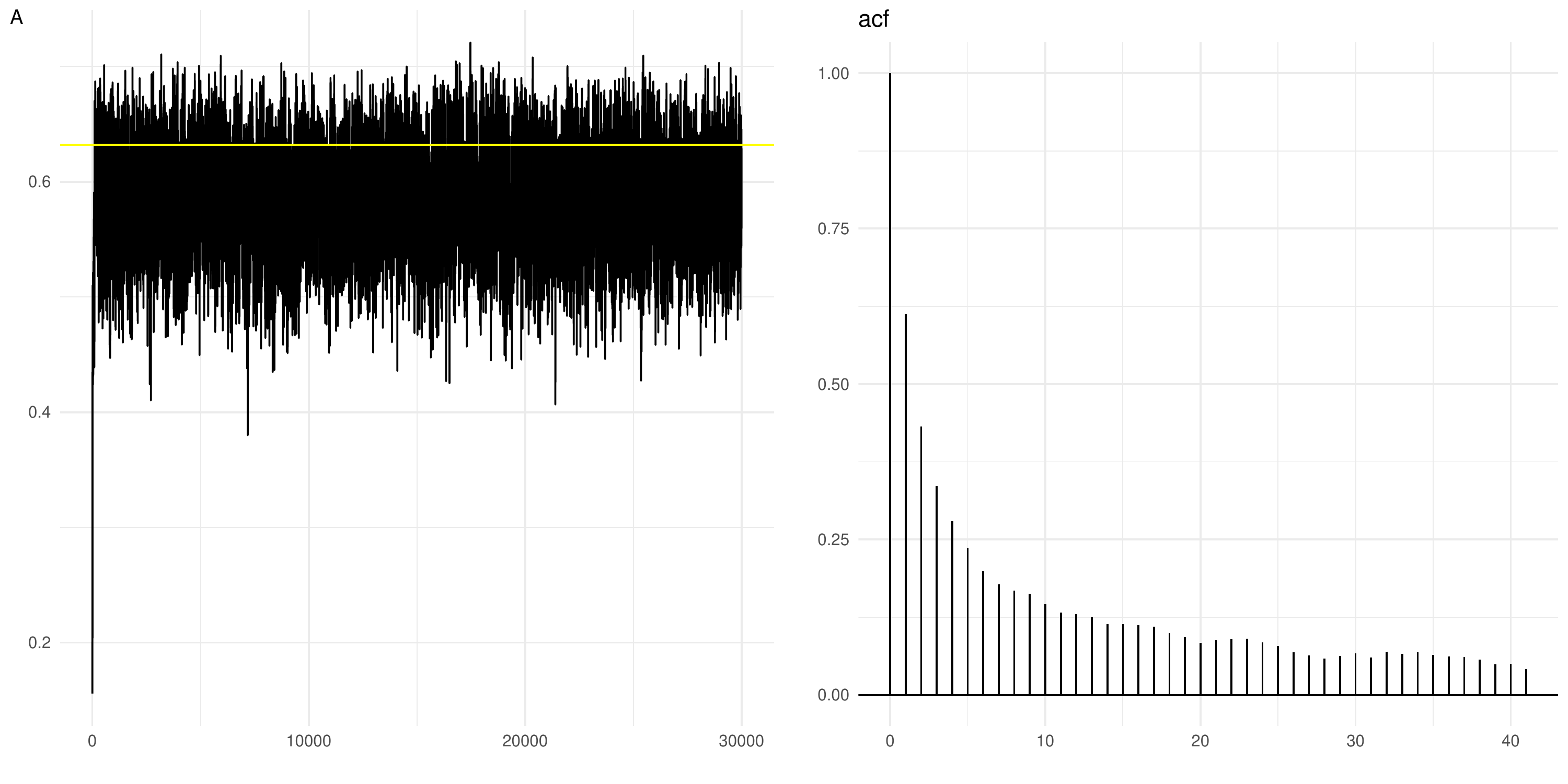}
}
\caption{Traceplot (left) and autocorrelation plot (right) for the  posterior distribution function evaluated at $1$ using the  algorithm detailed in Section \ref{sec:alg}. The horizontal line in the left-hand figure  depicts the true value
$F_0(1)=1-e^{-1}$. \label{fig:itacf} }
\end{figure}

We compute the posterior mean estimator for the function $F_0$ using equation (\ref{eq:postmean}) for two samples from the standard exponential distribution: one with sample size $50$ and the other with sample size $500$.
Figure \ref{fig:bayes} shows the results.  The total number of MCMC iterations was chosen to be $30,000$, with $15,000$ burn-in iterations.
\begin{figure}[!htp]
\centering
{
\includegraphics[scale=0.6]{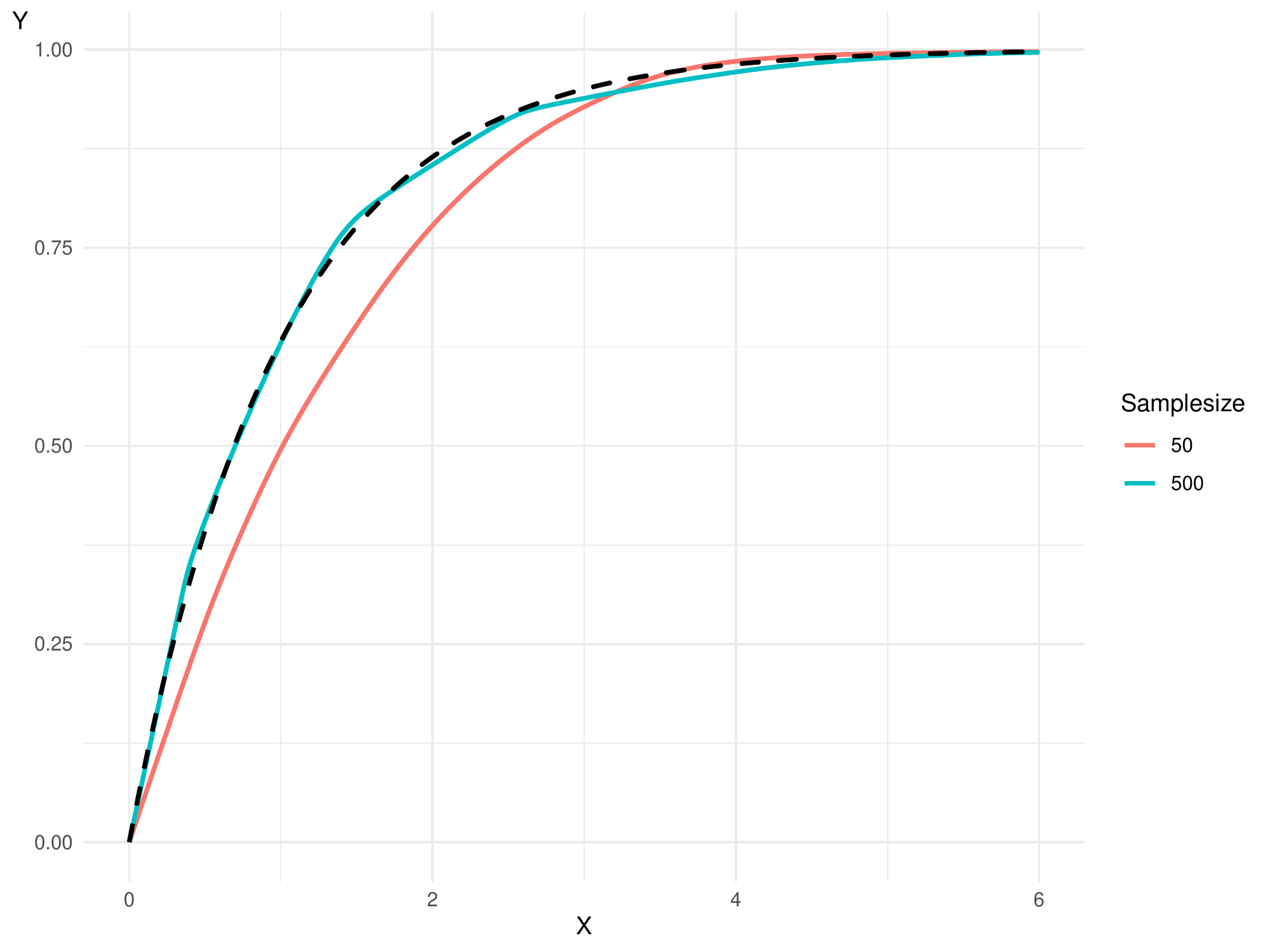}
}
\caption{Posterior mean in case the data are sampled from the standard exponential distribution. The two solid lines depict the posterior mean based on sample size either $50$ or $500$. The dashed  curve depicts the true  distribution function.
The total number of MCMC iterations was chosen to be $30.000$, with $15.000$ burn-in iterations.}
\label{fig:bayes}
\end{figure}

\medskip

We now compare different estimation methods:
\begin{itemize}
  \item the posterior mean for a concave distribution function;
  \item the maximum likelihood estimator under concavity;
  \item the maximum likelihood estimator without shape constraints.
\end{itemize}
We took $n=500$ and considered $K_i=1$, $K_i=2$ for $i=1,\dots,n$ (interval censoring case 1 and 2) and $K_i$ independently sampled from the discrete uniform distribution on the integers $\{1,2,\dots,20\}$, which we denote by $K\sim\mbox{Unif}(1,20)$.

We use the same prior specification as before. Figure \ref{fig:compare} depicts the estimators $\hat{F}$ (here we have three estimators: the NPMLE using the algorithm in \cit{WellnZhan}, the concave MLE studied in \cit{DumJong06} and the Bayesian posterior mean estimator)
and error curves $\hat{F}-F_0$, where $F_0$ is the true underlying distribution function. As the true distribution is smooth it is not surprising that NPMLE performs worst, as it is a step function. With an increasing number of inspection times,
the procedure of generating the inspection time and event time gives a narrow inspection interval for each event. Although the NPMLE does not consider the concavity assumption on $F_0$, it suggests a concave shape.
As can be seen in all cases, the concave MLE and the posterior mean estimator behave similarly.
\begin{figure}[!htp]
\centering
{
\includegraphics[scale=0.35]{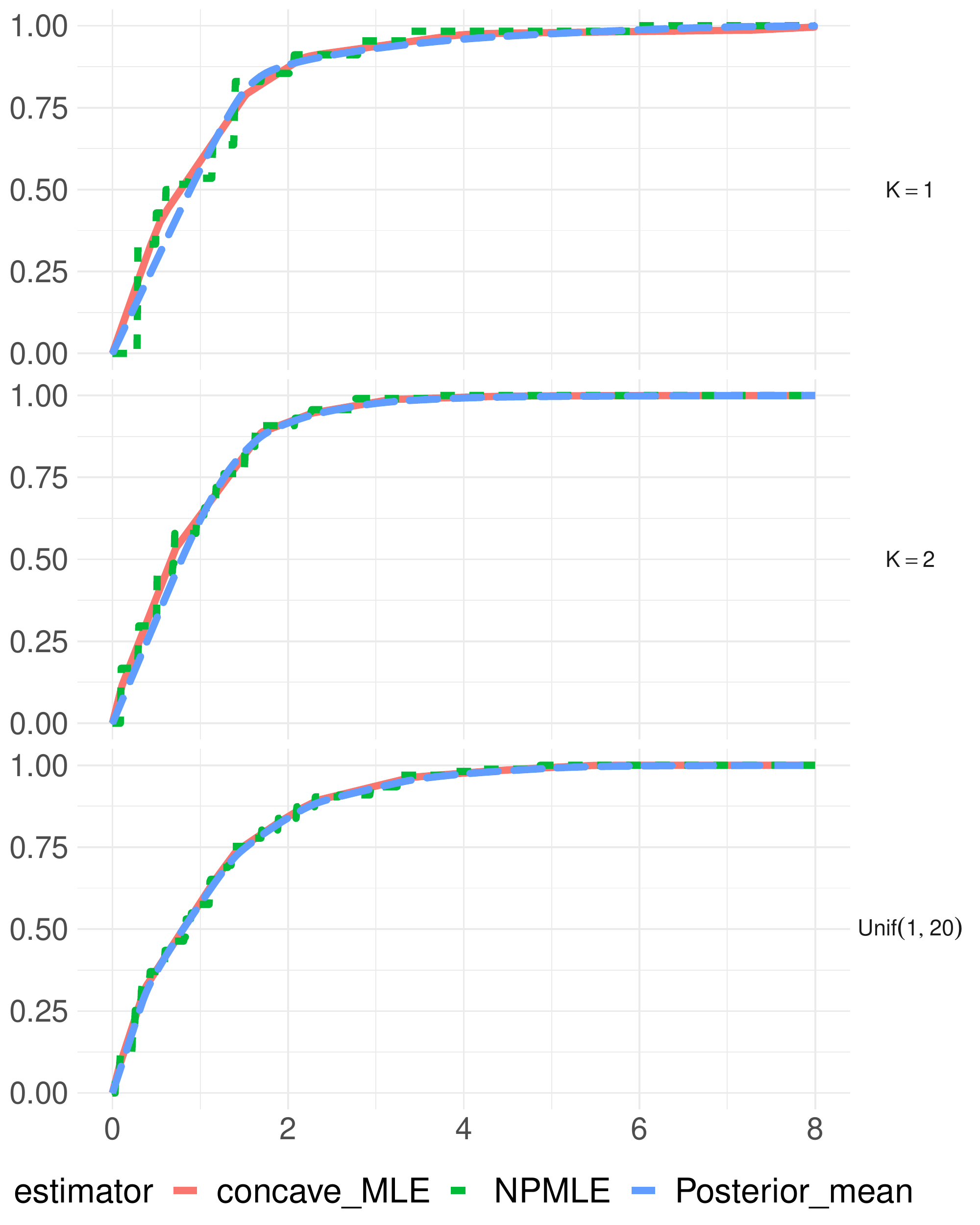}
\includegraphics[scale=0.35]{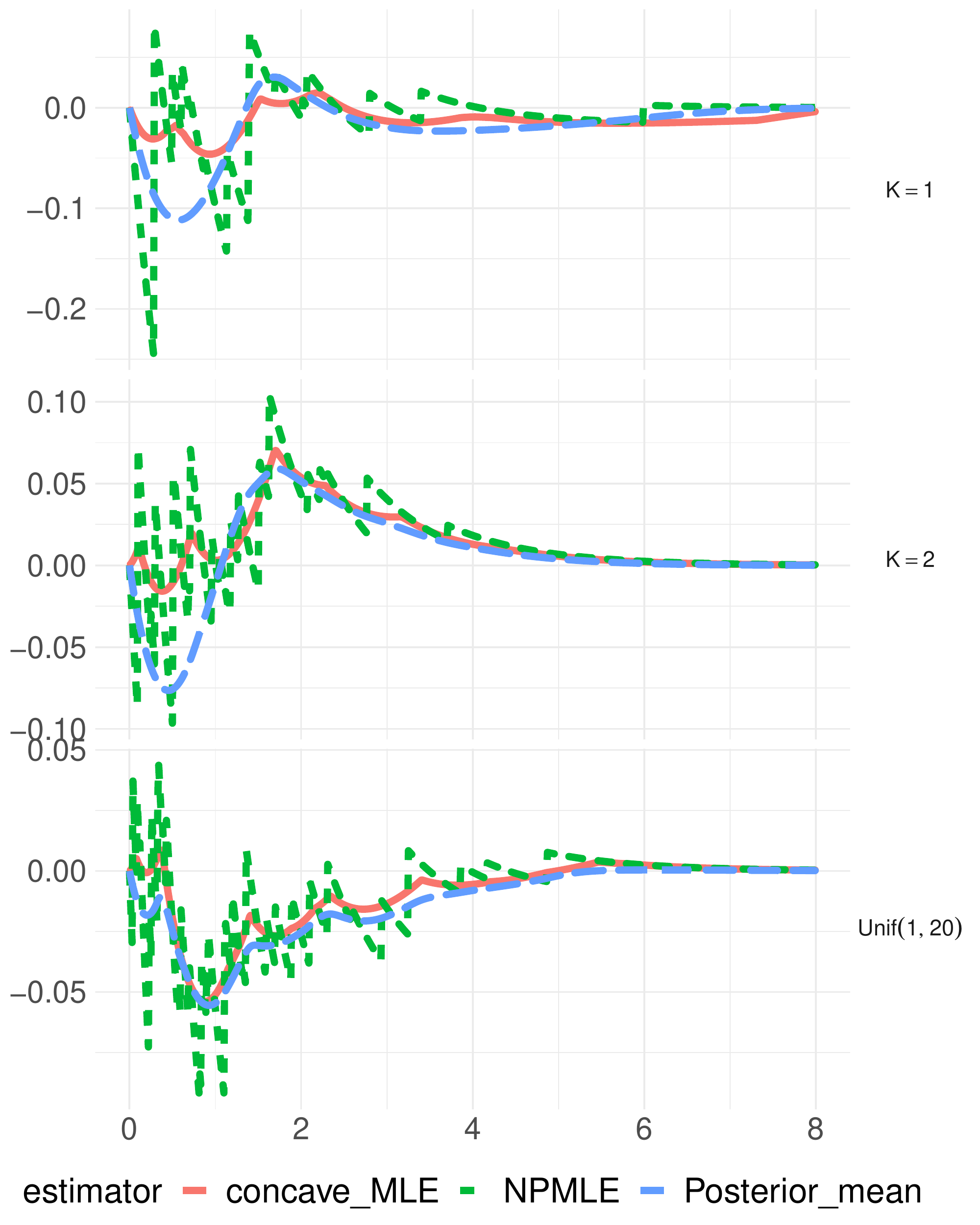}
}
\caption{Left: three cumulative distribution function estimators for $F_0$ (posterior mean, NPMLE and concave MLE). Right: corresponding error curves showing  $\hat{F}-F_0$. 
From top to bottom:  inspection time $K=1$, $K=2$ and $K\sim \mbox{Unif}(1,20)$. The data consists of $n=500$ independent draws  from the standard exponential distribution.}\label{fig:compare}
\end{figure}

Using the setting of mixed interval censoring ($K\sim \mbox{Unif}(1,20)$), we generated 50 data sets of sizes $n=50,100,200,400,800$ from the standard exponential and half-normal distribution and computed the NPMLE, the concave MLE, the posterior mean for each of the cases. Fix grid points
$t_j=j/100,j=1,\dots,m$, where we took $m=800$.
Figures \ref{fig:gridMSE_exp} and \ref{fig:gridMSE_halfnorm} show the log of the mean square error of $\hat{F}$ evaluated at $t=t_j, j=1,\dots,m$ for each sample size $n$, that is
\[\log R(\hat{F},F)(t)=\log \frac1{50}\sum_{k=1}^{50}(\hat{F}^{(k)}(t)-F(t))^2\]
where $\hat{F}^{(k)}$ represent estimator based on the $k-$th data set. We see that all three estimators give small error. As seen from figure \ref{fig:compare}, it can be explained by the setting of how to generate
mixed interval censoring data. We see that the posterior mean gives smallest error when $t$ is small, whereas all three estimators are comparable when $t\in [1,4]$ of case $n=800$. Finally, the NPMLE performs best when $t$ is big based on the data sets sample from the half-normal distribution.
\begin{figure}[!htp]
\centering
{
\includegraphics[scale=0.75]{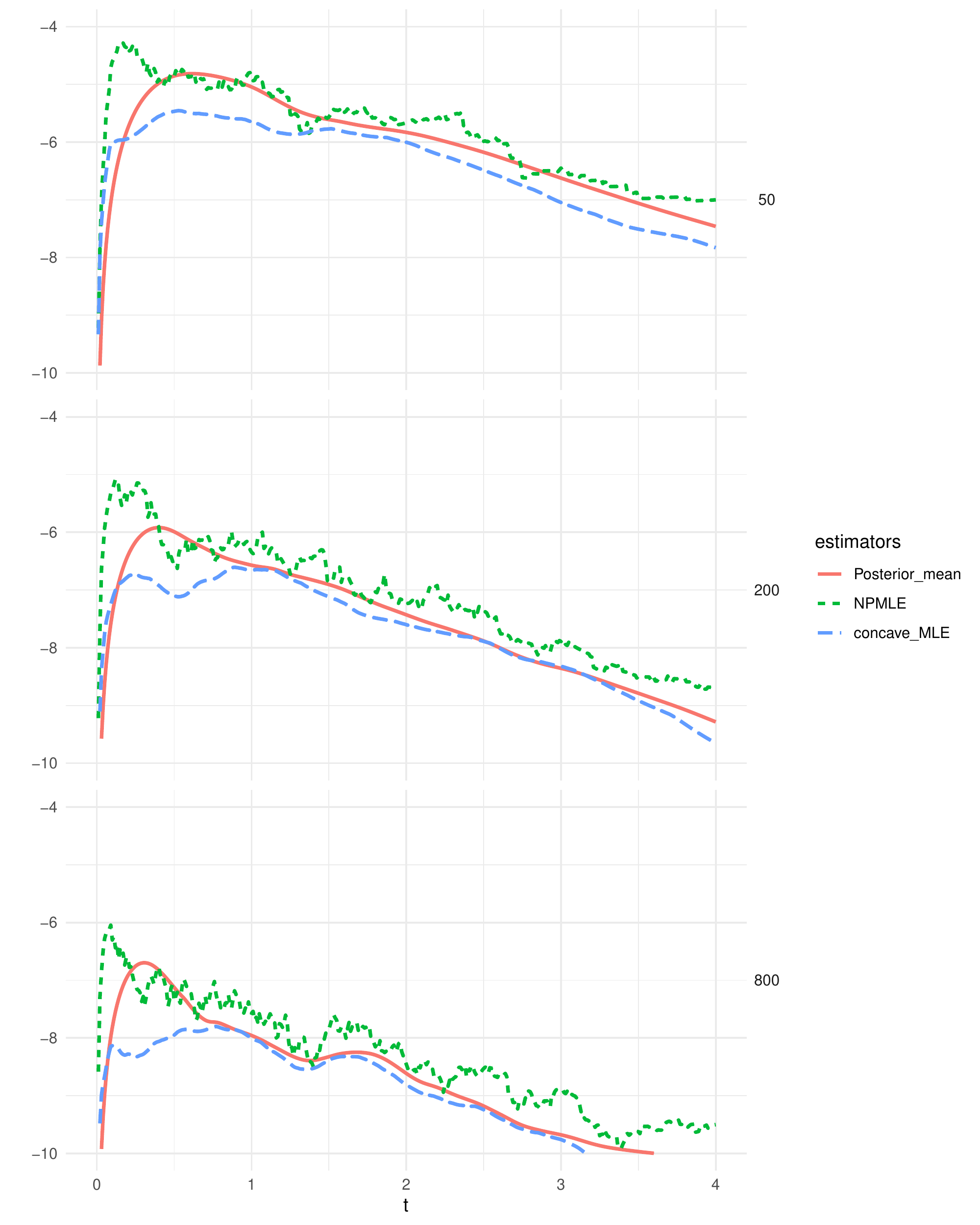}
}
\caption{The log mean square error ($\log R(\hat{F},F)$) evaluated at grid points $\{0.01,0.02,\dots,8.0\}$ for the NPMLE, the concave MLE and the posterior mean in 50 data sets of sample sizes $n=50$ (top), $n=200$ (middle) and $n=800$ (bottom). The  data generating distribution was taken to be 
the standard exponential distribution.}
\label{fig:gridMSE_exp}
\end{figure}

\begin{figure}[!htp]
\centering
{
\includegraphics[scale=0.75]{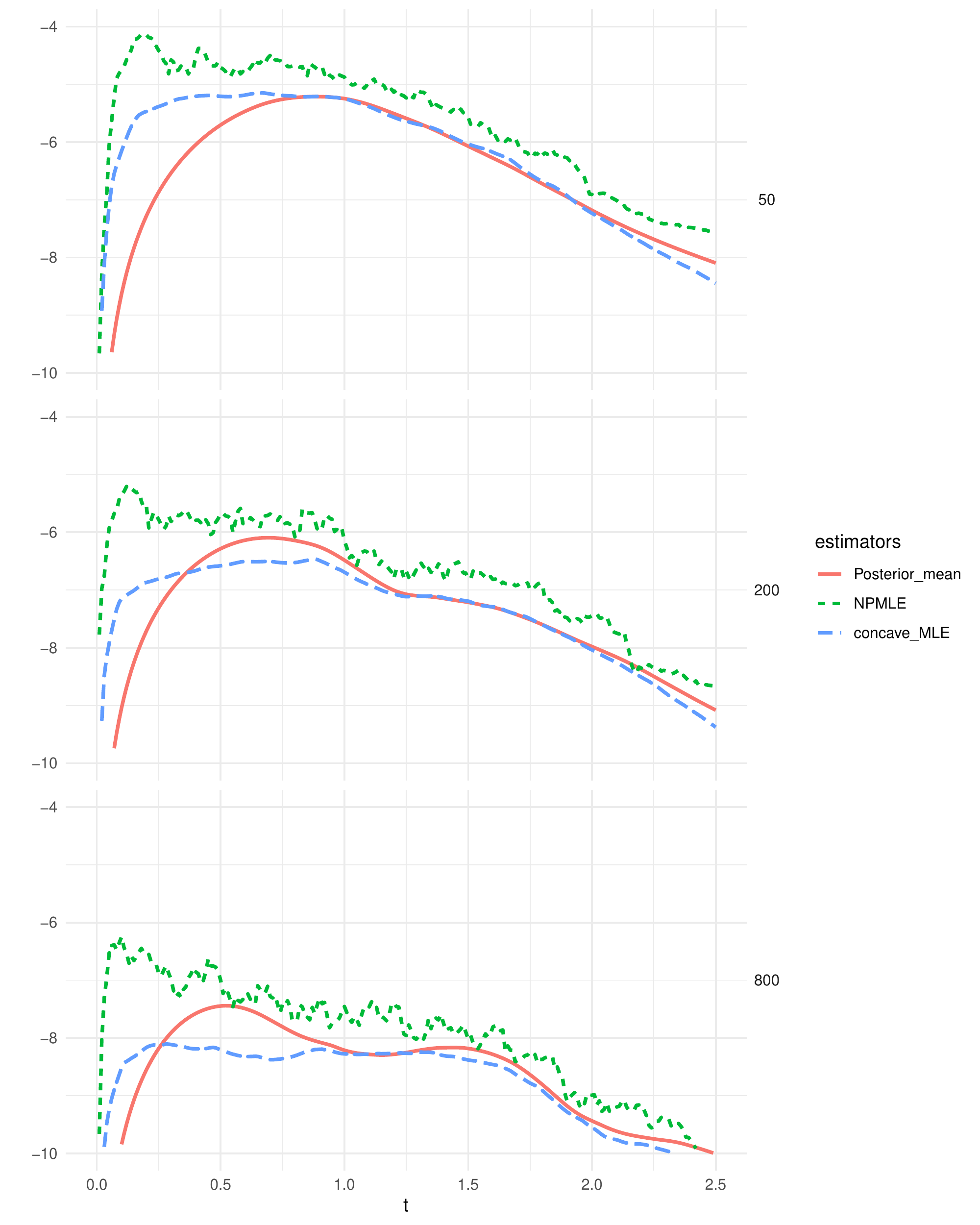}
}
\caption{The log mean square error ($\log R(\hat{F},F)$) evaluated at grid points $\{0.01,0.02,\dots,8.0\}$ for the NPMLE, the concave MLE and the posterior mean in 50 data sets of sample sizes $n=50$ (top), $n=200$ (middle) and $n=800$ (bottom). The  data generating distribution was taken to be the 
standard halfnormal distribution.}
\label{fig:gridMSE_halfnorm}
\end{figure}

We also consider a global value, the integrated square errors:
\[ISE^{(k)}(\hat{F},F)=\frac1{m}\sum_{j=1}^m(\hat{F}^{(k)}(t_j)-F(t_j))^2\]
for each sample size $n$, where $\hat{F}^{(k)}$ represent estimator based on the $k-$th data set, $k=1,\dots,50$. 
Figure \ref{fig:compareMISE} shows the mean of integrated square errors. In most of the cases, we see that the concave MLE has the smallest mean integrated square error,  The posterior mean laying between NPMLE and the concave MLE and close to the concave MLE in case of half-normal distribution.
\begin{figure}[!htp]
  \centering
  \includegraphics[scale=0.8]{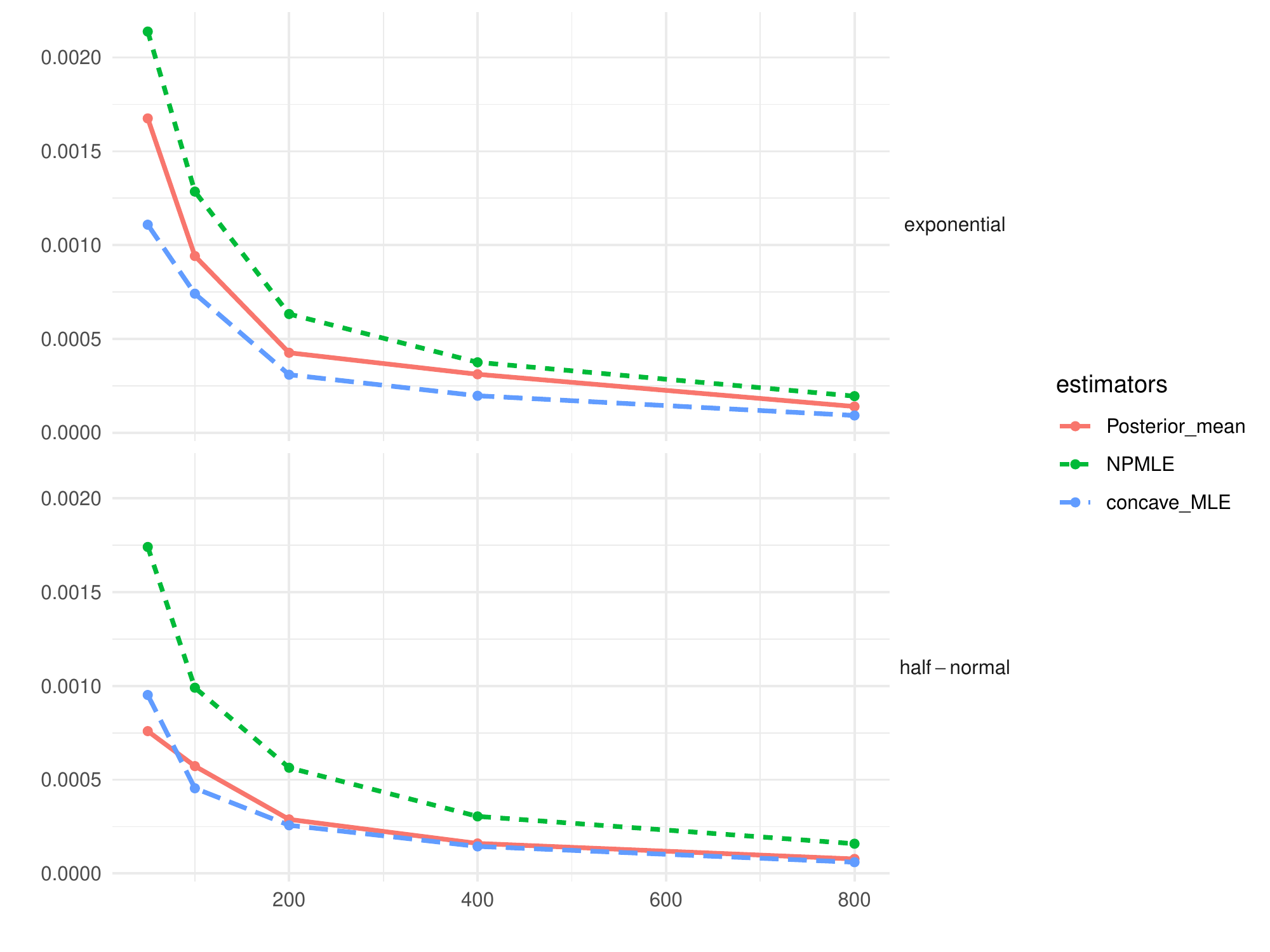}
\caption{The mean of $ISE^{(k)}(\hat{F},F)$ for the NPMLE, the concave MLE and the posterior mean in $50$ data sets of sample size $n \in \{50, 100, 200, 400, 800\}$ (corresponding to the horizontal axis).  The  data generating distribution was taken to be either 
the standard exponential distribution (top) or the standard halfnormal (bottom) distribution.} 
\label{fig:compareMISE}
\end{figure}

\section{Case study}
\label{sec:realdata}
In this section we illustrate the applicability of our method in real data examples. Using a nonparametric frequentist approach, producing confidence bands for the underlying distribution usually needs quite some fine tuning (see e.g. \cit{GroJo}). Contrary to the frequentist approach, within the Bayesian approach it is simple to construct  pointwise credible regions from MCMC output. We applied
 the Bayesian approach and two frequentist estimators to the Rubella data and Breast cancer data sets.

\begin{ex}\label{rubella}
Rubella is a highly contagious childhood disease. 
The Rubella data concerns the prevalence of rubella in $n=230$ Austrian males (see for more information \cit{Keiding(1996)}). The male individuals included in the data set represent an unvaccinated population. The data records
whether a person got infected or not before a certain time. Here  the upper limit of a persons's life span is set equal to $100$. Because there is only one inspection time per person, the data are actually case 1 interval censored.
Figure \ref{fig:vis_rubella} visualises the data, showing that the time intervals either start at $0$ or end at $100$.

 The settings for computing the posterior mean are as described in the previous section (DP as the prior, with concentration parameter $\alpha=1$ and the
 mixture of Pareto as the base measure. The  total number of iterations was set to $30.000$ where the initial  $15.000$ iterations have been treated as burn.
Figure \ref{fig:quantiles_rubella} shows the  three estimators and $95\%$ pointwise credible sets for the underlying distribution function. The mle (assuming the distribution function to be concave) is comparable with the posterior mean. However, the posterior mean provides a  smoother estimator as it is obtained by averaging and not as a maximizer of a likelihood (both the mle and mle under concavity assumption only change slope at censoring times).

\begin{figure}[!htp]
\centering
{
\includegraphics[scale=0.6]{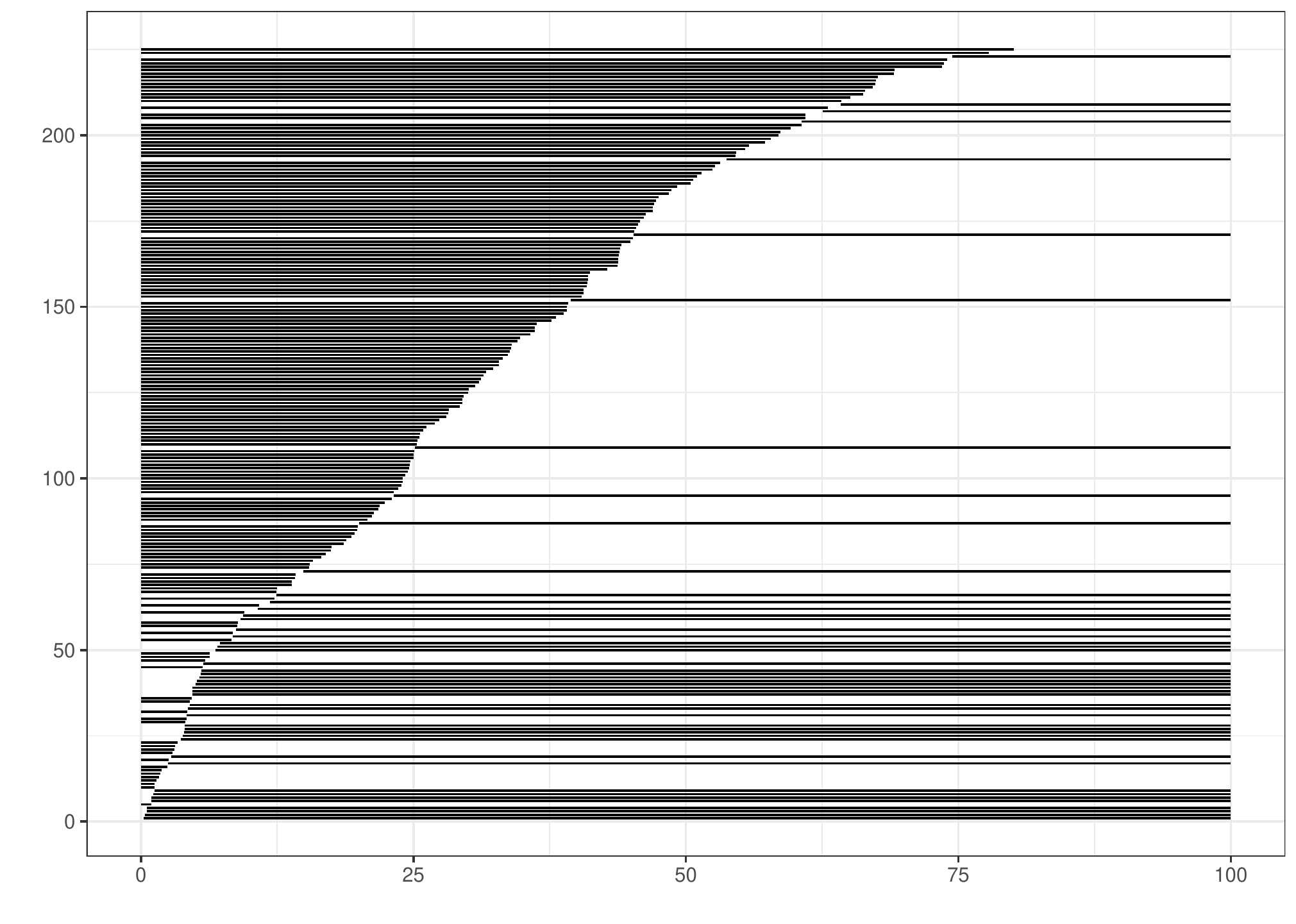}
}
\caption{Visualisation of Rubella data. The x-axis is the range of event time. The horizontal lines display the time intervals.}\label{fig:vis_rubella}
\end{figure}

\begin{figure}[!htp]
\centering
{
\includegraphics[scale=0.8]{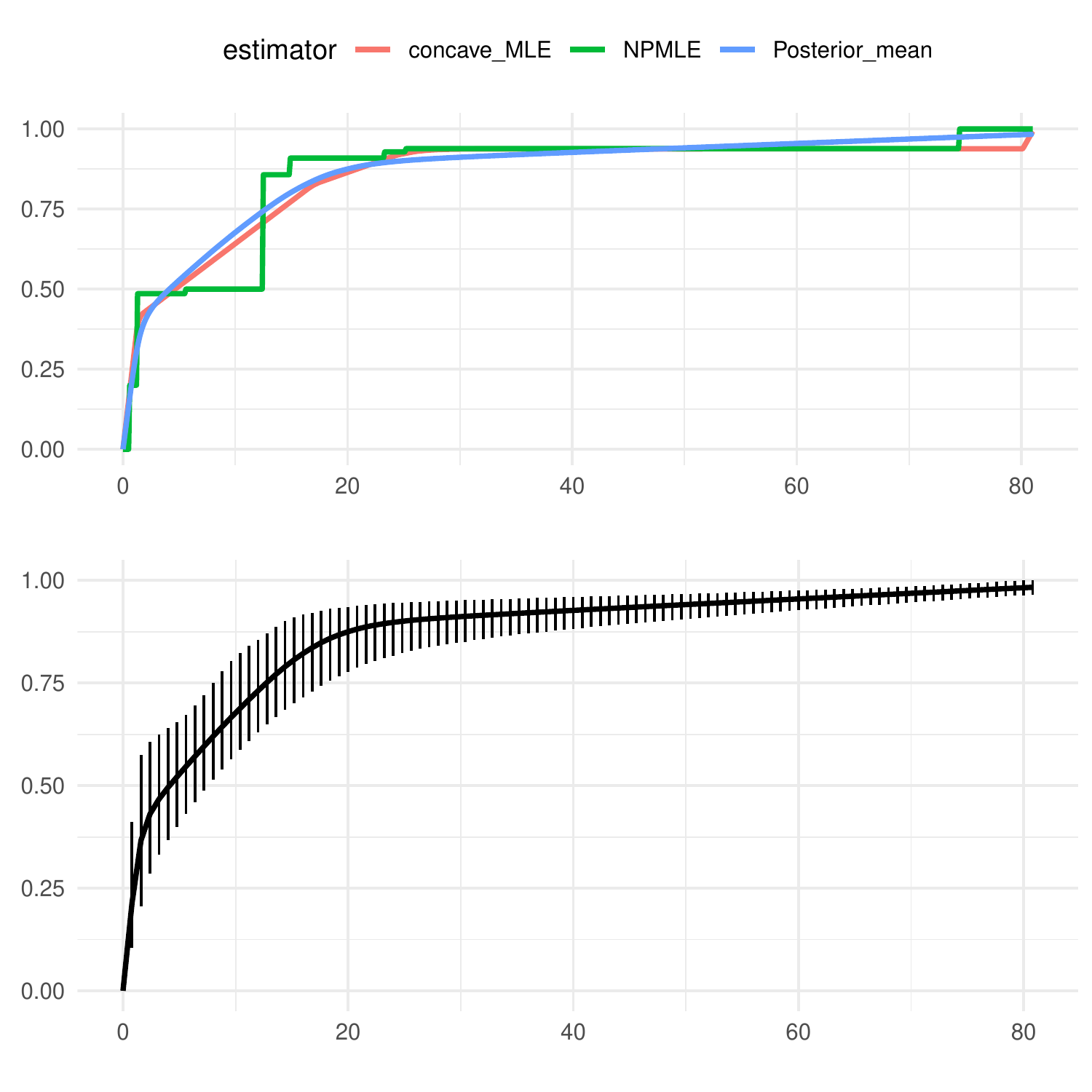}
}
\caption{Rubella data. Top: NPMLE, concave MLE and posterior mean estimators. Bottom:  $95\%$ pointwise credible sets of the estimated posterior mean for the underlying distribution function }\label{fig:quantiles_rubella}
\end{figure}
\end{ex}

\begin{ex}\label{bcos}
In the Breast cancer study discussed in \cit{FinkWol}, 94 early breast cancer patients were given radiation therapy with (RCT, 48) or without (RT, 46) adjuvant chemotherapy between 1976 and 1980. They were supposed to be seen at clinic visits every 4 to 6
months. However, actual visit times differ from patient to patient, and times
between visits also vary. In each visit, physicians evaluated the appearance breast retraction.
The data contain information about the time to breast retraction, hence, interval censored.
Figure \ref{fig:vis_bcos} visualises the data, we use the right end point $100$ for the right censoring case.
\begin{figure}[!htp]
\centering
{
\includegraphics[scale=0.35]{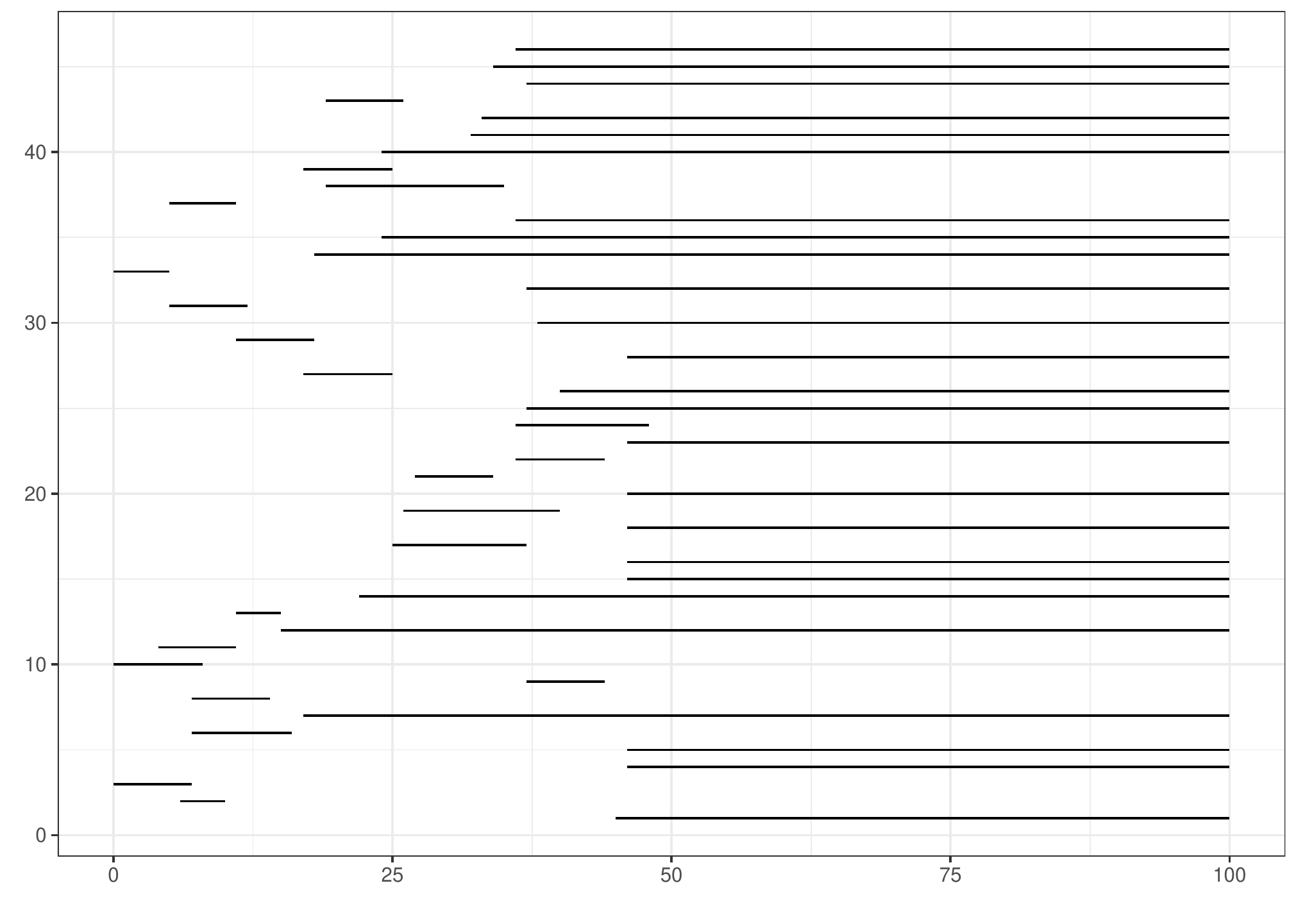}
\includegraphics[scale=0.35]{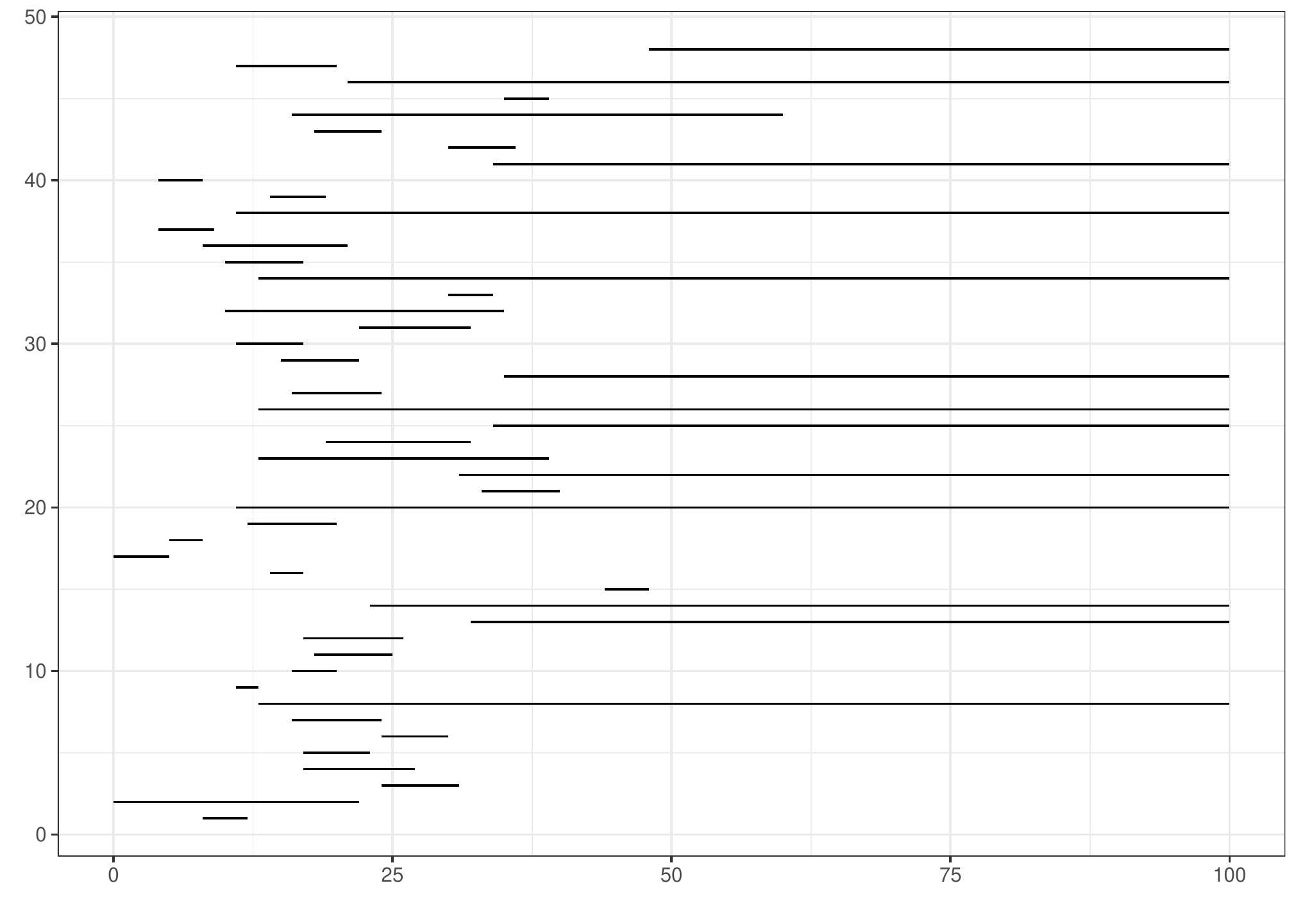}
}
\caption{Visualisation of the Breast cancer data (left: RT, right: RCT). The x-axis is the range of event times. The horizontal lines display the time intervals.}\label{fig:vis_bcos}
\end{figure}

The settings for computing the posterior mean are as in example \ref{rubella}.
Figure \ref{fig:quantiles_bcos} shows the three estimators under two treatments (RT and RCT) and $95\%$ credible sets for the underlying survival function.

\begin{figure}[!htp]
\centering
{
\includegraphics[scale=0.9]{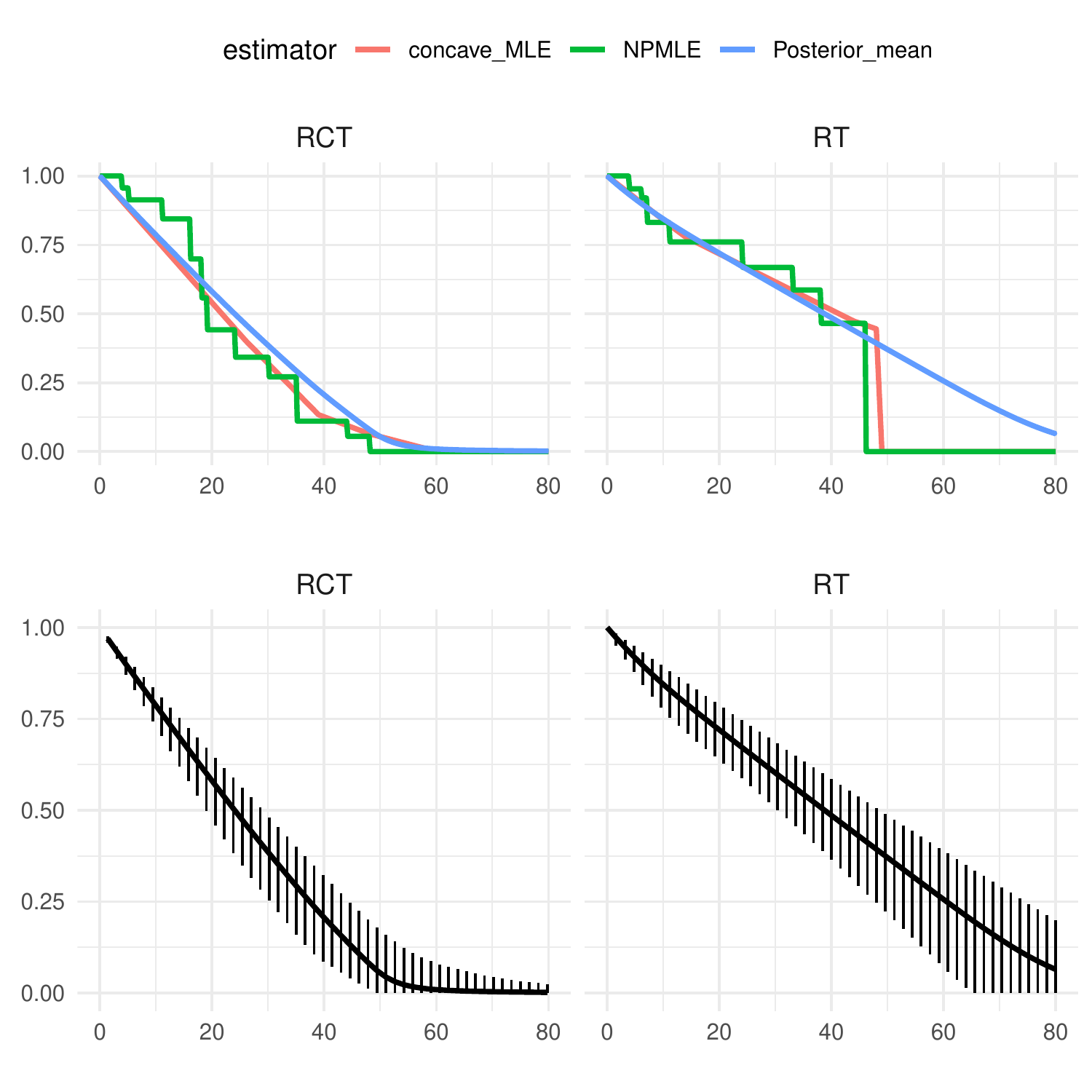}
}
\caption{Breast cancer data. Top: NPMLE, concave MLE and posterior mean estimators. Bottom:  $95\%$ pointwise credible sets of the estimated posterior mean for the underlying survival function. Left: treatment RCT. Right: treatment RT.}\label{fig:quantiles_bcos}
\end{figure}
\end{ex}

\section{Appendix: proofs of technical results}
\label{sec:app_paper2}
In the proof of lemma \ref{lem:KL}, we use the following lemma, it constructs a sequence of approximations for $F_0$.
\begin{lem}\label{lem:approximate}
Let $F_0$ satisfy the conditions stated in theorem \ref{thm:consis}. Then there exists a sequence of piecewise linear concave distribution functions $(F_m)$ such that
\[\sum_{k=1}^\infty p_K(k)\int g_k(t)h_{k,F_0,F_m}(t)dt\to 0\qquad as \quad m\to\infty.\]
\end{lem}
\begin{proof}
Since $F_0$ is a concave distribution function, its density $f_0$ is decreasing on $[0,\infty)$. We start off with the construction of  functions $f_m$ that approximate $f_0$ (Cf. Theorem 18 in \cit{Wu}). Choose $m\in\mathbb{N}$
and let $\tilde{f}_{0,m}=\frac{f_01_{\{[0,m]\}}}{F_0(m)}$, then $\tilde{f}_{0,m}\to f_0$ pointwise as $m\to\infty$.
Let $a_1$ and $a_2$ be real numbers such that $f_0(0)>a_1>a_2>0$. By the continuity of $f_0$, there exists $x_2>x_1$ satisfying $\tilde{f}_{0,m}(x_1)=a_1$ and $\tilde{f}_{0,m}(x_2)=a_2$. See also Figure \ref{fig:approx}.
Let $m_1\in\mathbb{N}$ and $m_2\in\mathbb{N}$ satisfy $\frac{m_1}{m}< x_1\le \frac{m_1+1}{m}$ and $\frac{m_2}{m}< x_2\le \frac{m_2+1}{m}$.
Then define
\[\tilde{f}_m(x)=\begin{cases}\tilde{f}_{0,m}(\frac{i}{m}), \qquad&\frac{i-1}{m}<x\le\frac{i}{m}, 1\le i\le m_1\\
a_1, \qquad&  \frac{m_1}{m}<x\le\frac{m_1+1}{m}\\
\tilde{f}_{0,m}(\frac{i-1}{m}), \qquad&\frac{i-1}{m}<x\le\frac{i}{m}, m_1+1<i\le m^2 .\end{cases}\]
and $\tilde{f}_m(0)=\tilde{f}_{0,m}(m^{-1})$. Because $f_0$ is continuous on $[0,m]$, $\tilde{f}_m$ converges pointwise to $f_0$  as $m\to\infty$. Note $\tilde{f}_m$ is not a probability density function, as it will not
integrate to one. We now normalize $\tilde{f}_m$ to a density function $f_m$. First we can rewrite $\tilde{f}_m$ as
$$\tilde{f}_m(x)=\sum_{i=1}^{m^2} \tilde{w}_i\phi(x,i/m),$$
where $\phi$ is defined as (\ref{phi}) and
\[\tilde{w}_i=\begin{cases}\frac{i}{m}(\tilde{f}_0(\frac{i}{m})-\tilde{f}_0(\frac{i+1}{m})), \qquad& 1\le i<m_1\\
\frac{m_1}{m}(\tilde{f}_0(\frac{m_1}{m})-a_1), \qquad&  i=m_1\\
\frac{m_1+1}{m}(a_1-\tilde{f}_0(\frac{m_1+1}{m})), \qquad&  i=m_1+1\\
\frac{i}{m}(\tilde{f}_0(\frac{i-1}{m})-\tilde{f}_0(\frac{i}{m})), \qquad& m_1+1<i< m^2\\
m\tilde{f}_0(\frac{m^2-1}{m}),\qquad& i=m^2.\end{cases}\]
 Let
$$w_i=\begin{cases}\tilde{w}_i\frac{1-\sum_{j=1}^{m_1-1}\tilde{w}_j-\sum_{j=m_2+1}^{m^2}\tilde{w}_j}{\sum_{j=m_1}^{m_2}\tilde{w}_j}, \qquad & m_1\le i\le m_2,\\
\tilde{w}_i,\qquad & \text{otherwise}.\end{cases}$$
Then $\sum_{i=1}^{m^2} w_i=1$ and $w_i\ge 0$ (for $m$ sufficiently large). Finally, define a sequence of probability density functions
\begin{equation}\label{eq:approx}f_m(x)=\sum_{i=1}^{m^2} w_i\phi(x,i/m).\end{equation}
Note that for $x\ge x_2$, $f_m(x)=\tilde{f}_m(x)$.
For each $x\in [0,m]$,
\begin{align*}
|f_m(x)-\tilde{f}_m(x)| &=\left|\sum_{i=1}^{m^2}w_i\phi(x,\frac{i}{m})-\sum_{i=1}^{m^2}\tilde{w}_i\phi(x,i/m)\right|\\
&=\left|\sum_{i=m_1}^{m_2}(w_i-\tilde{w}_i)\phi(x,i/m)\right|\\
&=\left|\left(\frac{1-\sum_{j=2}^{m_1-1}\tilde{w}_j-\sum_{j=m_2+1}^{m^2}\tilde{w}_j}{\sum_{j=m_1}^{m_2}\tilde{w}_j}-1\right)\sum_{i=m_1}^{m_2}\tilde{w}_i\phi(x,i/m)\right|\\
&\le \left|\frac{1-\sum_{j=2}^{m_1-1}\tilde{w}_j-\sum_{j=m_2+1}^{m^2}\tilde{w}_j}{\sum_{j=m_1}^{m_2}\tilde{w}_j}-1\right|\left(\sum_{i=m_1}^{m_2}\tilde{w}_i\right)\frac{m}{m_1}\\
&=\left|1-\sum_{i=2}^{m_1-1}\tilde{w}_i-\sum_{i=m_2+1}^{m^2}\tilde{w}_i-\sum_{i=m_1}^{m_2}\tilde{w}_i\right|\frac{m}{m_1}\\
&=\Big|1-\frac1{m}\sum_{i=2}^{m^2} \tilde{f}_{0,m}(i/m)-\frac{a_1}{m}\Big|\frac{m}{m_1} \to 0
\end{align*}
Here we use that $m/m_1\to x_1^{-1}$ and that the expression within the modular signs converges to 0 as difference between $\int_0^m\tilde{f}_{0,m}(x)dx$ and its Riemann sum approximate. Then we have $|f_m-\tilde{f}_{0,m}|\to0$  pointwise
and $\tilde{f}_{0,m}\to f_0$ pointwise. Hence $f_m$ is a decreasing density and converges to $f_0$ pointwise. See an example in figure \ref{fig:approx} for visualize $f_0,\tilde{f}_m$ and $f_m$.
\begin{figure}[!htp]
\centering
\includegraphics[width=12cm,height=8cm]{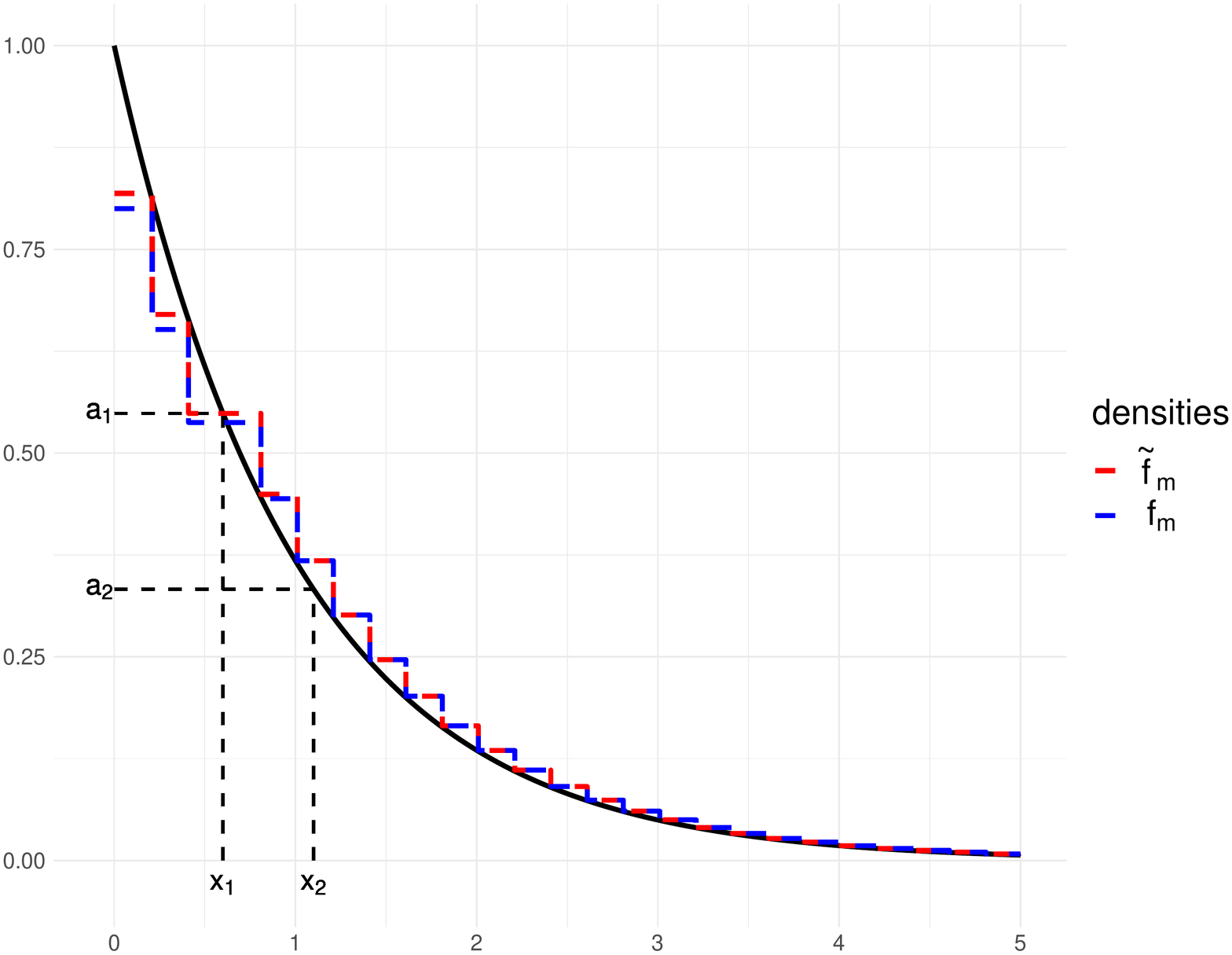}
\caption{Approximation for a decreasing function $f_0$. First we construct a step function $\tilde{f}_m$, then we normalize the weights $\tilde{w}_i$ to $w_i$ such that $f_m$ defined by (\ref{eq:approx}) is a decreasing density function.}
\label{fig:approx}
\end{figure}

Define $F_m(x)=\int_0^xf_m(t)dt$, then using dominated convergence, we have $F_m\to F_0$ pointwise. As $m\to\infty$ ($m>L$), then it follows that for all $k$ and $t$
$$h_{k,F_0,F_m}(t)=\sum_{j=1}^{k+1}(F_0(t_j)-F_0(t_{j-1}))\log\frac{F_0(t_j)-F_0(t_{j-1})}{F_m(t_j)-F_m(t_{j-1})}\to 0.$$
The next step is to find an integrable upper bound for $|h_{k,F_0,F_m}|$. Denote $p_j=F_0(t_{j})-F_0(t_{j-1})$ for $j=1,\dots,k+1$ and note that $\sum_{j=1}^{k+1}p_j=1$.
Then
\[| h_{k,F_0,F_m}(t)|\le\sum_{j=1}^{k+1}p_j|\log p_j|+\sum_{j=1}^{k+1}p_j|\log(F_m(t_{j})-F_m(t_{j-1}))|.\]
Using Lagrange multipliers, the first sum achieves its maximal value over all probability vectors when all $p_j$'s would be equal. Hence it can be bounded by $\log(k+1)$. For the second sum,
by the construction of $f_m$ we know that when $x<x_2$, $f_m(x)\ge f_m(x_2)=\tilde{f}_m(x_2)=\tilde{f}_{0,m}(m_2/m)\ge a_2$; when $ x_2\le x\le m$, $f_m(x)=\tilde{f}_m(x)\ge \tilde{f}_{0,m}(x)=\frac{f_0(x)}{F_0(m)}\ge f_0(x)$.
Since there exists $j_0\in\{1,\dots,k+1\}$ such that $t_{j_0-1}<x_2\le t_{j_0}$, the second sum can be bounded by $I_1+I_2+I_3$, where
\begin{align*}
I_1&=-\sum_{j=1}^{j_0-1}p_j\log(a_2(t_j-t_{j-1}))\\
I_2&=-p_{j_0}\log(a_2(x_2-t_{j_0-1})+F_0(t_{j_0}\wedge m)-F_0(x_2))\\
I_3&=-\sum_{j=j_0+1}^{k}p_j\log p_j+p_{k+1}|\log(F_0(m)-F_0(t_k))|
\end{align*}
Again using the Lagrange multipliers, we have
\begin{align*}
I_1&=-\sum_{j=1}^{j_0-1}\frac{p_j}{a_2(t_j-t_{j-1})}(a_2(t_j-t_{j-1}))\log(a_2(t_j-t_{j-1}))\\
&\le -\frac{M}{a_2}\sum_{j=1}^{j_0-1}(a_2(t_j-t_{j-1}))\log(a_2(t_j-t_{j-1}))
\le \frac{M}{a_2}\log k
\end{align*}
In the second step we use $p_j\le M(t_j-t_{j-1})$. In the final step, we use that $\sum_{j=1}^{j_0-1}a_2(t_j-t_{j-1})\le 1$. To bound $I_2$, we know that $-x\log x\le \frac1{e}$ when $x\in(0,1]$. Splitting $I_2$ into two parts, we have
\begin{align*}
I_2&\le-(F_0(t_{j_0})-F_0(x_2))\log(F_0(t_{j_0}\wedge m)-F_0(x_2))-(F_0(x_2)-F_0(t_{j_0-1}))\log(a_2(x_2-t_{j_0-1}))\\
&\le \frac1{e}\left(\frac{F_0(t_{j_0})-F_0(x_2)}{F_0(t_{j_0}\wedge m)-F_0(x_2)}+\frac{F_0(x_2)-F_0(t_{j_0-1})}{a_2(x_2-t_{j_0-1})}\right)\\
&\le \frac1{e}\left(\frac{F_0(t_{j_0})}{F_0(t_{j_0}\wedge m)}+\frac{M}{a_2}\right)
\le \frac1{e}\left(\frac1{F_0(L)}+\frac{M}{a_2}\right)\\
\end{align*}
In the last step, we used that  $\frac{F_0(t_{j_0})}{F_0(t_{j_0}\wedge m))}\le\max(1,1/F_0(m))\le 1/F_0(L)$. Similarly, we can bound $I_3$ by
\begin{align*}
I_3&\le\log k+p_{k+1}|\log(F_0(m)-F_0(t_k))|\\
&\le \log k+\frac1{e}\frac{1-F_0(t_k)}{F_0(m)-F_0(t_k)}
\le \log k+\frac1{e}\frac1{F_0(m)}\le \log k+\frac1{e}\frac1{F_0(L)}
\end{align*}
Therefore, having these bounds we obtain
$$\mid h_{k,F_0,F_m}(t)\mid\le \left(\frac{M}{a_2}+2\right)\log (k+1)+\frac1{e}\left(\frac{2}{F_0(L)}+\frac{M}{a_2}\right).$$
By the assumption in theorem \ref{thm:consis}, we have $\EE\log(K+1)\le C(r)K^r\infty<\infty$ for some constant $C(r)$ depend on $r$, hence
 $$\sum_{k=1}^\infty p_K(k)\int g_k(t)\mid h_{k,F_0,F_m}(t)\mid dt< \infty.$$
Therefore, by the dominated convergence theorem,
$$\sum_{k=1}^\infty p_K(k)\int g_k(t)h_{k,F_0,F_m}(t)dt\to 0.$$
\end{proof}

\subsection{Proof of lemma \ref{lem:KL}}
\begin{proof} By lemma \ref{lem:approximate}, for any $\eta>0$ there exists a sequence of piecewise linear concave distribution functions $(F_m)$ such that
\begin{equation}\label{F0m}\sum_{k=1}^\infty p_K(k)\int g_k(t)h_{k,F_0,F_m}(t)dt<\eta/2\end{equation}
for all $m$ big enough. Recall definition (\ref{eq:approx}), $f_m(x)=\sum_{i=1}^{m^2} w_i\phi(x,\frac{i}{m})=\int \phi(x,\theta)dP_m(\theta)$, where $P_m(\cdot)=\sum_{i=1}^{m^2}w_i\delta_{i/m}(\cdot)$. Without loss of generality,
assume $w_i>0$ for all $i=1,\dots,m^2$.
Given $m$ fixed, for some $0<\epsilon<\min(1,e^{\eta/4}-1)$, define a discrete probability measure $P'_{m,\epsilon}(\cdot)=\sum_{i=1}^{m^2}w_i\delta_{(i+\epsilon/2)/m}(\cdot)$. Moreover, define
the bounded Lipschitz distance on the set of probability measure on $[0,\infty)$ by
\[d_{BL}(P,Q)=\sup_{\psi\in \mathcal{C}_1}\left|\int\psi dP-\int\psi dQ\right|,\]
where $\mathcal{C}_1$ denotes the set of Lipschitz continuous functions on $[0,\infty)$ with Lipschitz constant 1. Then $d_{BL}$ induces the weak topology (See Appendix A.2 in \cit{GhoVaart}).
Choose $0<\delta\le\frac{\epsilon}{4m}(1-e^{-\eta/4})\min_{1\le i\le m^2}w_i$ and define the open set
$$\Omega_m=\left\{P\in\mathcal{M}: d_{BL}(P,P'_{m,\epsilon})<\delta\right\}.$$
 Choose Lipschitz continuous functions $\psi_j, j=1,\dots,m$ with compact support $[\frac{j}{m},\frac{j+\epsilon}{m}]$, satisfying $\psi_j(\theta)=\frac{\epsilon}{4m}$ if
 $\theta\in(\frac{j+\frac1{4}\epsilon}{m},\frac{j+\frac{3}{4}\epsilon}{m})$ and $0\le \psi_j\le\frac{\epsilon}{4m}$.
Denote $U_j=[\frac{j}{m},\frac{j+\epsilon}{m}]$, $j=1,\dots,m^2$. Then for any $P\in\Omega_m$, $j=1,\dots,m^2$, we have
\[\left|\int \psi_jdP-\int \psi_jdP'_{m,\epsilon}\right|\le d_{BL}(P,P'_{m,\epsilon})<\delta.\]
It also follows that for $j=1,\dots,m^2$,
\begin{align*}\frac{\epsilon}{4m}P(U_j)&\ge \int\psi_jdP\ge \int\psi_jdP'_{m,\epsilon}-\delta\\
&\ge\frac{\epsilon}{4m}\int_{(j+\frac1{4}\epsilon)/m}^{(j+\frac{3}{4}\epsilon)/m}1dP'_{m,\epsilon}-\delta
=\frac{\epsilon}{4m}w_j-\delta\ge \frac{\epsilon}{4m}e^{-\eta/4}w_j.\end{align*}
That is $P(U_j)\ge e^{-\eta/4}w_j$, for $j=1,\dots,m^2$. Using this lower bound and the mixture representation (\ref{eq:mixpre}), we have for any $x\ge 0$, $P\in\Omega_m$,
\[\frac{f_m(x)}{f_P(x)}\le \frac{\sum_{i=1}^{m^2} w_i\phi(x,\frac{i}{m})}{\sum_{i=1}^{m^2} \int_{U_i}\phi(x,\theta)dP(\theta)}
\le \frac{\sum_{i=1}^{m^2} w_i\frac{m}{i}1_{\{x\le \frac{i}{m}\}}}{\sum_{i=1}^{m^2} \frac{m}{i+\epsilon}1_{\{x\le \frac{i}{m}\}}P(U_j)}
\le (1+\epsilon)e^{\eta/4}\le e^{\eta/2}.
\]
As this implies \[F_m(t_j)-F_m(t_{j-1})=\int_{t_{j-1}}^{t_j}f_m(x)dx\le e^{\eta/2}\int_{t_{j-1}}^{t_j}f_P(x)dx=e^{\eta/2}(F_P(t_j)-F_P(t_{j-1})),\] we have that
\begin{align}h_{k,F_m,F_P}(t)&=\sum_{j=1}^{k+1}(F_0(t_j)-F_0(t_{j-1}))\log\frac{F_m(t_j)-F_m(t_{j-1})}{F_P(t_j)-F_P(t_{j-1})}\nonumber\\
&\le \frac{\eta}{2}\sum_{j=1}^{k+1}(F_0(t_j)-F_0(t_{j-1}))
\le \eta/2.\label{FmP}\end{align}
Note that $h_{k,F_0,F_P}(t)=h_{k,F_0,F_m}(t)+h_{k,F_m,F_P}(t)$. Combining inequalities (\ref{F0m}) and (\ref{FmP}), we have
 $$\sum_{k=1}^\infty p_K(k)\int g_k(t)h_{k,F_0,F_P}(t)dt<\eta.$$
That means $\{F_P\in\mathcal{F}:P\in\Omega_m\}\subset S(\eta)$. 
Since $\Omega_m$ is an open weak neighborhood of $P'_m$ in the neighborhood a and support$(\Pi^\ast)=\mathcal{M}$, we have $\Pi^\ast(\Omega_m)>0$.

Recall that the prior $\Pi$ on $\mathcal{F}$ is induced by the prior $\Pi^\ast$ on $\mathcal{M}$ and the mixture representation (\ref{eq:mixpre}),
therefore $\Pi(S(\eta))\ge\Pi^\ast(\Omega_m)>0$.

\end{proof}

\subsection{Proof of lemma \ref{lem:test}}
\begin{proof}
We construct a test function depending on  data $\mathcal{D}_n$. For any $\epsilon>0$, define the event $A_n=\{d_n(\hat{F}_n,F_0)\ge \epsilon/2\}$,
where $\hat{F}_n$ is the maximum likelihood estimator of the underlying distribution based on observations $\mathcal{D}_n$ (see Theorem 3 in \cit{DumJong06}) and $d_n$ is defined as (\ref{lossfunc}).
Define $\Phi_n=1\{A_n\}$, then as $n\to\infty$,
\begin{align}
\EE_0\Phi_n&=\EE_{K,T}\{\EE_{F_0}[\Phi_n|K,T]\}\nonumber\\
&=\EE_{K,T}\{\PP_{F_0}[d_n(\hat{F}_n,F_0)\ge \epsilon/2|K,T]\}\to0
\label{firsttype}
\end{align}
The final step holds because the consistency of $\hat{F}_n$, $\PP_{F_0}[d_n(\hat{F}_n,F_0)\ge \epsilon/2|K,T]\to 0$ and this probability is bounded by 1.
Similarly, given $(K,T)$, for all $F\in U_\epsilon$
\begin{align*}
\EE_F[1-\Phi_n|(K,T)]
&=\PP_F[\{d_n(\hat{F}_n,F_0)\le\epsilon/2\}\cap \{d_n(F,F_0)> \epsilon\}|(K,T)]\\
&\le\PP_F[d_n(F_0,F)-d_n(\hat{F}_n,F_0)\ge\epsilon/2|(K,T)]\\
&\le\PP_F[d_n(F,\hat{F}_n)\ge\epsilon/2|(K,T)]
\end{align*}
Then it is sufficient to prove for any $\epsilon>0$,
\[\EE_{(K,T)}\left\{\sup_{F\in U_\epsilon}\PP_F[d_n(F,\hat{F}_n)>\epsilon|(K,T)]\right\}\to 0.\]
We state that
\begin{equation}\label{secondtype}\sup_{F\in U_\epsilon}\PP_F[d_n(F,\hat{F}_n)>\epsilon|(K,T)]\to 0.\end{equation}
Then (\ref{firsttype}) and (\ref{secondtype}) are equivalent to the existence of a uniformly exponentially consistent test for testing $H_0:F=F_0$ versus $H_1:F\in U_{\epsilon}$ (see Proposition 4.4.1 in \cit{GhoRam}).

Now we show the inequality (\ref{secondtype}) holds. For a fixed $F\in\mathcal{F}$, the consistency result in \cit{DumJong06} claims that $d_n(F,\hat{F}_n)\to_{p}0$, 
Actually, they proved that $\PP_F[d_n(F,\hat{F}_n)>\epsilon]\to 0$ given the censoring times $(K,T)$. We checking all steps of the proof in \cit{DumJong06}, the consistency is follows from the finite expectation of $K$ and the bound $F\le 1$.
Define
$$H^2(F,G)=(2n)^{-1}\sum_{i,j}(F_{i,j}-G_{i,j})^2.$$
The consistency result is follows from the following steps:
\begin{enumerate}
\item $d_n(F,\hat{F}_n)\le 8^{1/2}H(F,\hat{F}_n)$ ;
\item $H(F,\hat{F}_n)^2\le n^{-1}\sum_{i,j} (\Delta_{i,j}-F_{i,j})(\hat{F}_{n,i,j}/F_{i,j})^{1/2}$;
\item $n^{-1}\sum_{i,j} (\Delta_{i,j}-F_{i,j})(\hat{F}_{n,i,j}/F_{i,j})^{1/2}\le \sup_{G\in\mathcal{F}}|\sum_i (\psi_i(G)-\EE_F \psi_i(G))|$;
\end{enumerate}
where 
$\psi_i(G)=n^{-1}\sum_j\Delta_{i,j}(G_{i,j}/F_{i,j})^{1/2}$. Hence, it is sufficient to show
\[\PP_F\left\{\sup_{G\in\mathcal{F}}\Big|\sum_i (\psi_i(G)-\EE_F \psi_i(G))\Big|>\epsilon\right\}\to 0.\]
By theorem \ref{thm:Pollard}, this is a consequence of the following conditions: for some sequences $\delta_n\to0, b_n\to 0$,
\begin{align}
&\EE_F\sum_{i=1}^n\sup_{G\in\mathcal{F}}|\psi_i(G)|=O(1),\label{B1}\\
&\EE_F\sum_{i=1}^n \mathbf{1}\{\sup_{G\in\mathcal{F}}|\psi_i(G)|>\delta_n\}\sup_{G\in\mathcal{F}}|\psi_i(G)|=b_n,\label{B2}\\
&\text{for any}\, u>0,\quad \log\mathcal{N}(u,\mathcal{F},\rho_n)\le c(u).\label{B3}
\end{align}
where
\[\mathcal{N}(u,\mathcal{F},\rho_n)=\min\left\{\#\mathcal{G}: \mathcal{G}\subset\mathcal{F},\inf_{G'\in\mathcal{G}}\rho_n(G,G')\le u\,\text{for all}\, G\in\mathcal{F}\right\},\]
and
\[\rho_n(G,G')=\sum_{i=1}^n|\psi_i(G)-\psi_i(G')|.\]
We first give the main inequalities to derive these conditions.
For (\ref{B1}), \[\EE_F\sum_{i=1}^n\sup_{G\in\mathcal{F}}|\psi_i(G)|\le n^{-1}\sum_i (K_i+1)^{1/2}.\]
For (\ref{B2}), \[\EE_F\sum_{i=1}^n1\{\sup_{G\in\mathcal{F}}|\psi_i(G)|>\delta_n\}\sup_{G\in\mathcal{F}}|\psi_i(G)|\le n^{-1}\sum_i(K_i+1)^r(n\delta_n)^{-2\kappa}\to 0,\]
where $\kappa\in(0,\frac1{2})$, recall that $EK^r<\infty$ and choosing $n\delta_n\to\infty$.
As for (\ref{B3}), $\rho_n$ can be bounded by a finite measure, hence \[\log\mathcal{N}(u,\mathcal{F},\rho_n)\le Cu^{-1}\] for some constant $C$. (For more details see the proof of Theorem 3 in \cit{DumJong06}). Hence,
\[b_n=n^{-1}\sum_i(K_i+1)^r(n\delta_n)^{-2\kappa}, c(u)=Cu^{-1}.\]
By equation (\ref{eq:upbound}), we have
\[\PP_F\left\{\sup_{G\in\mathcal{F}}\Big|\sum_i (\psi_i(G)-\EE_F \psi_i(G))\Big|>\epsilon\right\}\le4\epsilon^{-1}b_n+ 128C\epsilon^{-1}\epsilon^{-1}\exp\left(-\frac{\epsilon^2}{512n\delta_n^2}\right)\]
Note that the right side do not depend on $F$, hence the inequality (\ref{secondtype}) holds.

\end{proof}

\subsection{A technical result for proving uniform convergence}
The following theorem follows from theorem 8.2 in \cit{Pollard}.
\begin{thm}\label{thm:Pollard}
Let $f_1(w,t),f_2(w,t),\dots,f_n(w,t)$ be independent processes with integrable envelopes $F_1(w)$, $F_2(w),\dots, F_n(w)$. If for each $\epsilon>0$,
\begin{enumerate}
\item there is a sequence $\delta_n\to 0$ such that
\[ \frac1{n}\sum_{i=1}^n\EE F_i\mathbf{1}\{F_i>\delta_n\}<\epsilon,\quad\mbox{for all}\quad n,\]
\item $\log N(u,\mathcal{F}_{nw},\rho_n)=c(u)$,
\end{enumerate}
then \[\sup_t\left|\sum_{i=1}^n(f_i(w,t)-\EE f_i(w,t))\right|\to 0 \quad\mbox{in probability.}\]
Here $\mathcal{N}(u,\mathcal{F}_{nw},\rho_n)$ is the covering number of $\mathcal{F}_{nw}$ with distance
\[\rho_n=\rho_n(t,t')=\sum_{i=1}^n|f_i(w,t)-f_i(w,t')|.\]
\end{thm}

\begin{proof}
Define event $A_{n,i}:=\{F_i>\delta_n\}$, then we split the expectation into two parts:
\begin{align*}\sup_t\Big|\sum_i (f_i(w,t)-\EE f_i(w,t))\Big|
&\le\sup_t\Big|\sum_i (f_i(w,t)\mathbf{1}\{A_{n,i}\}-\EE f_i(w,t)\mathbf{1}\{A_{n,i}\})\Big|\\
&\,\,+\sup_t\Big|\sum_i (f_i(w,t)\mathbf{1}\{A^c_{n,i}\}-\EE f_i(w,t)\mathbf{1}\{A^c_{n,i}\})\Big|
\end{align*}
For the first item in the right side, by the condition 1, we have
\begin{equation}\label{bound1}
\begin{split}
\PP&\left\{\sup_t\Big|\sum_i (f_i(w,t)\mathbf{1}\{A_{n,i}\}-\EE f_i(w,t)\mathbf{1}\{A_{n,i}\})\Big|>\epsilon/2\right\}\\
&\quad\quad\quad\le 2\epsilon^{-1}\EE\left\{\sup_t\Big|\sum_i (f_i(w,t)\mathbf{1}\{A_{n,i}\}-\EE f_i(w,t)\mathbf{1}\{A_{n,i}\})\Big|\right\}\\
&\quad\quad\quad\le 4\epsilon^{-1}\EE\left\{\sum_i \sup_t f_i(w,t)\mathbf{1}\{A_{n,i}\}\right\}
=4\epsilon^{-1} b_n
\end{split}
\end{equation}
For the second item, denote $f^\ast_i=f_i \mathbf{1}\{A^c_{n,i}\}$. Using symmetrization, we have
\[\PP\left\{\sup_t\Big|\sum_i (f^\ast_i(w,t)-\EE f^\ast_i(w,t))\Big|>\epsilon/2\right\}
\le4\EE_{\sigma}\PP\left\{\sup_t\Big|\sum_i \sigma_if^\ast_i(w,t)\Big|>\epsilon/8\right\},\]
where $\sigma_i=1$ or $-1$ with probability $1/2$ independently. By the definition of covering number $\mathcal{N}(\epsilon/16,\mathcal{F}_{nw},\rho_n)$,
given $w$, for each $t$ in $\mathcal{F}_{nw}$, there exists $t'$ such that the distance $\rho_n(t,t')\le \epsilon/16$. Then we have
\begin{align}\PP\left\{\sup_t\Big|\sum_i \sigma_if^\ast_i(w,t)\Big|>\epsilon/8\right\}
&\le\PP\left\{\max_{t'}\Big|\sum_i \sigma_if^\ast_i(w,t')\Big|+\rho_n(t,t')>\epsilon/8\right\}\nonumber\\
&\le \PP\left\{\max_{t'}\Big[\sum_i \sigma_if^\ast_i(w,t')\Big]>\epsilon/16\right\}\nonumber\\
&\le \mathcal{N}(\epsilon/16,\mathcal{F}_{nw},\rho_n)\max_{t'}\PP\left\{\Big|\sum_i \sigma_if^\ast_i(w,t')\Big|>\epsilon/16\right\}
\label{bound2}
\end{align}
By the Hoeffding's inequality and $f^\ast_i(w,t')\le\delta_n$, we further have
\begin{equation}\label{bound3}
\PP\left\{\Big|\sum_i \sigma_if^\ast_i(w,t')\Big|>\epsilon/16\right\}
\le2\exp\left(-\frac{2(\epsilon/16)^2}{\sum_i(2f^\ast_i(w,t'))^2}\right)
\le2\exp\left(-\frac{\epsilon^2}{512n\delta_n^2}\right)
\end{equation}
Therefore, combining inequalities (\ref{bound1}), (\ref{bound2}) and (\ref{bound3}), we have
\begin{equation}\label{eq:upbound}\PP\left\{\sup_t\Big|\sum_i (f_i(w,t)-\EE f_i(w,t))\Big|>\epsilon\right\}\le 4\epsilon^{-1}b_n+ 8c(\epsilon/16)\epsilon^{-1}\exp\left(-\frac{\epsilon^2}{512n\delta_n^2}\right).\end{equation}
By choosing $n\delta_n^2\to 0$, we have the right side tend to 0.
\end{proof}

\end{document}